\documentclass{article}
\usepackage{graphicx}
\usepackage{authblk}
\usepackage{hyperref}
\hypersetup{
	pdftitle={Stochastic Augmented Lagrangian Method in Riemannian Shape Manifolds},
	pdfauthor={Caroline Geiersbach, Tim Suchan, and Kathrin Welker}
}
\usepackage[T1]{fontenc}
\usepackage[utf8]{inputenc}
\usepackage{lmodern} 

\providecommand{\keywords}[1]
{
	\small	
	\textbf{\textit{Keywords. }} #1
}
\providecommand{\ams}[1]
{
	\small	
	\textbf{\textit{AMS subject classifications. }} #1
}

\usepackage{amsmath,amssymb,amscd,latexsym} 
\usepackage{amsthm}
\usepackage{upgreek}
\usepackage{bbm} 
\usepackage{amsfonts}
\usepackage{mathrsfs}
\usepackage{cancel}
\usepackage{mathtools} 
\usepackage{stmaryrd} 
\usepackage{afterpage} 
\usepackage{enumitem} 

\usepackage{algorithm}
\usepackage[noend]{algpseudocode}

\usepackage{color}
\usepackage{graphicx}
\graphicspath{{figures/}{tikzs/}}
\usepackage{overpic}
\usepackage{subcaption}
\usepackage{placeins} 
\usepackage{footmisc}  
\usepackage{tablefootnote}
\newlength\figureheight 
\newlength\figurewidth
\usepackage{svg}
\svgpath{{svgs/}}
\usepackage{tikz}
\usetikzlibrary{shapes.geometric,calc,positioning}
\usepackage{tikz-cd}
\usepackage{tikz-3dplot}
\usepackage{pgfplots}				
\usepackage{varwidth}
\newlength\mysvgwidth
\newlength\boxwidth
\pgfplotsset{compat=newest,grid style={dotted,black}}
\usepgfplotslibrary{external} 			
\usepgfplotslibrary{patchplots}
\newcommand{\R}{{\mathbb{R}}}

\newcommand{\N}{{\mathbb{N}}}

\newcommand{\E}{{\mathbb{E}}}
\newcommand{\pP}{{\mathbb{P}}}

\renewcommand{\vec}[1]{\boldsymbol{#1}}
\newcommand{\Deriv}{\mathrm{d}_{\vec{u}}}

\DeclareMathOperator{\Divv}{div}
\newcommand{\Div}[1]{\Divv{#1}}

\DeclareMathOperator{\vol}{vol}
\DeclareMathOperator{\peri}{peri}
\DeclareMathOperator{\dist}{dist}

\renewcommand{\d}{\,\mathrm{d}}
\newcommand{\dx}{\,\mathrm{d} \vec{x}}
\newcommand{\ds}{\,\mathrm{d} \vec{s}}
\newcommand{\vx}{\vec{x}}
\newcommand{\vl}{\vec{\lambda}}
\newcommand{\vh}{\vec{h}}
\newcommand{\vw}{\vec{w}}
\newcommand{\velocity}{\vec{q}}
\newcommand{\pressure}{p}

\newcommand{\adjvelocity}{\vec{\varphi}}
\newcommand{\adjpressure}{\psi}

\newcommand{\vn}{\vec{n}}

\newcommand{\vu}{\vec{u}}
\newcommand{\vV}{\vec{V}}
\newcommand{\vK}{\vec{K}}
\newcommand{\cU}{\mathcal{U}}
\newcommand{\cG}{\mathcal{G}}

\definecolor{shape1}{RGB}{27,158,119}
\definecolor{shape2}{RGB}{217,95,2}
\definecolor{shape3}{RGB}{117,112,179}

\newcommand{\GD}{\varGamma_D}
\newcommand{\GN}{\varGamma_N}

\newcommand{\intD}[1]{\int_{D} #1 \dx}
\newcommand{\intDi}[1]{\int_{D_i} #1 \dx}
\newcommand{\intui}[1]{\int_{u_i} #1 \ds}

\newenvironment{data}
{\par\footnotesize}
{\par\addvspace{\bigskipamount}}

\newenvironment{acknowledgements}
{\par\footnotesize}
{\par\addvspace{\bigskipamount}}

\newcommand{\new}[1]{#1}
\newcommand{\rename}[1]{#1}

\newtheorem{theorem}{Theorem}[section]
\newtheorem{ass}{Assumption}
\newtheorem{definition}{Definition}[section]
\newtheorem{lemma}{Lemma}[section]
\newtheorem*{remark}{Remark}

\begin{document}
	
	\title{Stochastic Augmented Lagrangian Method in \new{Riemannian Shape Manifolds}}
	
	
	\author{Caroline Geiersbach}
	\affil{Weierstrass Institute, Mohrenstraße 39, 10117 Berlin, Germany
		\href{mailto:caroline.geiersbach@wias-berlin.de}{\ttfamily caroline.geiersbach@wias-berlin.de}\vspace*{6pt}}
	
	\author{Tim Suchan}
	\affil{Helmut-Schmidt-Universität / Universität der Bundeswehr Hamburg, Holstenhofweg 85, 22043 Hamburg, Germany, \href{mailto:suchan@hsu-hh.de}{\ttfamily suchan@hsu-hh.de}\vspace*{6pt}}
	
	\author{Kathrin Welker}
	\affil{Technische Universität Bergakademie Freiberg, Akademiestraße 6, 09599 Freiberg, Germany, \href{mailto:Kathrin.Welker@math.tu-freiberg.de}{\ttfamily Kathrin.Welker@math.tu-freiberg.de}}
	
	
	\maketitle
	
	\begin{abstract}
		In this paper, we present a stochastic augmented Lagrangian approach on (possibly infinite-dimensional)  Riemannian manifolds to solve sto\-chas\-tic optimization problems with a finite number of deterministic constraints.
		We investigate the convergence of the method, which is based on a stochastic approximation approach with random stopping combined with an iterative procedure for updating Lagrange multipliers.  The algorithm is applied to a multi-shape optimization problem with geometric constraints and demonstrated numerically.
	\end{abstract}
	
\newline
\ams{
	49Q10
	,
	60H35
	,
	35R15
	,
	49K20
	,
	41A25
	,
	60H15
	, 
	60H30
	,
	35R60
}
	
	\section{Introduction} \label{sec:intro}
	
	In this paper, we  concentrate on stochastic optimization problems of the form
	\new{
		\begin{equation}\tag{P}
			\label{eq:SO-problem-abstract}
			\begin{aligned}
				&\min_{u \in \mathcal{U}}\, \lbrace j(u):=\E[{J}(u,\vec{\xi})] = \int_\Omega J(u, \vec{\xi}(\omega)) \d \pP(\omega)\} \\
				&\text{subject to (s.t.)} \quad h_i(u) = 0 \quad i \in \mathcal{E}, \quad h_i(u) \leq 0 \quad i \in \mathcal{I}. 
			\end{aligned}
		\end{equation}
		Here, $\mathcal{U}$ is a Riemannian manifold and $\vec{\xi} \colon \Omega \rightarrow \Xi \subset \R^m$ is a random vector defined on a given probability space.
		We assume that we have deterministic constraints of the form $\vec{h}\colon\mathcal{U} \rightarrow \R^n$, $u \mapsto \vec{h}(u) = (h_1(u), \dots, h_n(u))^\top$}, where we distinguish between the index set $\mathcal{E}$ of equality constraints and the index set $\mathcal{I}$ of inequality constraints.

	Our investigations are motivated by applications in shape optimization, where an objective function is supposed to be minimized with respect to a shape, or a subset of $\R^{d}$.
	Finding a correct model to describe the set of shapes is one of the main challenges in shape optimization.
	From a theoretical and computational point of view, it is attractive to optimize in Riemannian manifolds because algorithmic ideas from \cite{Absil} can be combined with approaches from differential geometry as outlined in \cite{LoayzaGeiersbachWelkerHandbook}.
	\new{This is one of the main reasons why we focus on Riemannian manifolds in this paper. One needs to take into account that these Riemannian manifolds could be also infinite dimensional, e.g., the space of plane curves \cite{MichorMumford1,MichorMumford,MioSrivastavaJoshi,MichorMumfordShahYounes}, the space of piecewise-smooth curves 
		\cite{Pryymak2023}, and  the space of surfaces in higher dimensions 
		\cite{BauerHarmsMichor,BauerHarmsMichor_SobolevII,KMP,KurtekKlassenDingSrivastava,MichorMumford2}.
	}
	Often, more than one shape needs to be considered, which leads to so-called multi-shape optimization problems. As applications, we can mention electrical impedance tomography, where the material distribution of electrical properties such as electric conductivity and permittivity inside the body is examined \cite{Cheney1999,Kwon2002,laurain2016distributed}, 
	and the optimization of biological cell composites in the human skin \cite{SiebenbornNaegel,Siebenborn2017}.
	
	\new{If one focuses on one-dimensional shapes, the above-mentioned space of plane unparametrized curves is a prominent example of an infinite-dimensional manifold. In our numerical application (cf. Section \ref{sec:numericalResults}), we also focus on this shape space. Our choice of this space comes from the fact that in shape optimization, the set of permissible shapes generally does not allow a vector space structure. 
		One should note that there is no obvious distance measure without a vector space structure, which is a central difficulty in the formulation of efficient optimization methods. 
		If one cannot work in vector spaces, Riemannian shape manifolds are the next best option, but they come with additional difficulties; see Section~\ref{subsec:convergence}. 
	}

	A central difficulty in \eqref{eq:SO-problem-abstract} is that the constraints lead to a  stochastic optimization problem that cannot be handled using standard techniques such as gradient descent or Newton's method; additionally, the numerical solution of the problem may be intractable on account of the expectation. In this work, we propose a stochastic augmented Lagrangian method to solve problems of the form \eqref{eq:SO-problem-abstract}. The proposed method combines the smoothing properties of the augmented Lagrangian method with a reduction in complexity granted by stochastic approximation.

	The augmented Lagrangian method has been extensively studied; see \cite{Bertsekas1982,Birgin2014} for an introduction to the method \new{when $\mathcal{U} = \R^n$}. 
	Substantial theory can be found in the literature for PDE-constrained optimization, \new{which is related to our setting in PDE-constrained shape optimization and} where convergence has been studied in function spaces; see \cite{kanzow2018error,kanzow2019improved,kanzow2018augmented,Karl2018,Steck2018}. This theory does not apply even for deterministic counterparts of \eqref{eq:SO-problem-abstract} since \new{our control variable $u$} belongs to a Riemannian manifold, not a Banach space. \new{The study of constrained optimization on Riemannian manifolds is still nascent. There are relatively recent advances in first-order optimality conditions in KKT form, including the development of constraint qualifications analogous to the finite-dimensional setting \cite{bergmann2019intrinsic,yang2014}. The augmented Lagrangian method has recently been developed for Riemannian manifolds \cite{kovnatsky2016madmm,liu2020simple,yamakawa2022sequential}. These methods have been developed for deterministic problems, however, and therefore cannot be applied to problems of the form \eqref{eq:SO-problem-abstract}.}%
	
	Stochastic approximation is a class of algorithms that originated from the paper \cite{Robbins1951} and has developed in recent decades due to its applicability to high-dimensional stochastic optimization problems. \new{Thanks in part to applications in machine learning, these algorithms are increasingly being developed in the setting of Riemannian optimization; see, e.g., \cite{bonnabel2013stochastic,khuzani2017stochastic,sato2019riemannian,zhang2016riemannian,zhang2016first}.} The most basic algorithm is the stochastic gradient method, which can be used to solve an unconstrained version of (P), i.e., the problem of minimizing the expectation.
	Recently, the stochastic gradient method was proposed to handle PDE-constrained shape optimization problems  \cite{Geiersbach2021,LoayzaGeiersbachWelkerHandbook}. In \cite{Geiersbach2021}, asymptotic convergence was proven for optimization variables belonging to a Riemannian manifold and the connection was made to shape optimization following the ideas in \cite{Welker2016}. However, the stochastic gradient method cannot solve  problems of the form \eqref{eq:SO-problem-abstract}.

	While both augmented Lagrangian and stochastic approximation methods are well-developed, the combined method---what we call the stochastic augmented Lagrangian method---is not. In the context of training neural networks, a combined stochastic gradient/augmented Lagrangian approach in the same spirit as ours can be found in the paper \cite{Dener2020}. Our method, however, involves a novel use of the randomized multi-batch stochastic gradient method from \cite{Ghadimi2013,Ghadimi2016}, where a random number of stochastic gradient steps are chosen. We use this strategy to solve the inner loop optimization problem for fixed Lagrange multipliers and penalty parameters. A central consequence of the random stopping rule from \cite{Ghadimi2013,Ghadimi2016} is that convergence rates of the expected value of the norm of the gradient can be obtained, even in the nonconvex case. The random stopping rule in combination with an outer loop procedure can be used to adaptively adjust step sizes and batch sizes for a tractable algorithm where asymptotic convergence to stationary points of the original problem is guaranteed. 
	
	The paper is structured as follows. In Section \ref{sec:optimization-approach}, we present the stochastic augmented Lagrangian method for optimization on Riemannian manifolds and analyze its convergence.   
	Then, an application for our method is introduced and results of numerical tests are presented in Section \ref{sec:numericalResults}. \new{To conclude, we summarize our results in Section \ref{sec:conclusion}.}

	\section{Optimization Approach}
	\label{sec:optimization-approach}
	In this section, we introduce the stochastic augmented Lagrangian method  for Riemannian manifolds. In view of our later application to shape optimization, where convexity of the objective function $j$ cannot be expected, we focus on providing results for the nonconvex case. First, in Section \ref{subsec:background}, we will provide background material that will be of use in our analysis. 
	\new{In particular, definitions and theorems from differential topology and geometry that are required in this paper will be provided. For background details, we refer to, e.g., \cite{KrieglMichor,Kuehnel,Lang,Lee} for differential geometry and \cite{Gut2013} for probability theory.}
	The algorithm is presented in Section \ref{subsec:algorithm}. Convergence of the method is proven in two parts: in Section \ref{subsec:convergence}, we provide an efficiency estimate for the inner loop procedure, corresponding to a randomized multi-batch stochastic gradient method. Then, in Section \ref{subsec:outer-loop}, convergence rates with respect to the outer loop procedure, which corresponds to a stochastic augmented Lagrangian method, are given.

	\subsection{Background and Notation}
	\label{subsec:background}
	
	

	
	\new{We consider the Euclidean norm $\lVert \cdot\rVert_2$ on $\R^n$ throughout the paper.}  
	For a differentiable Riemannian manifold \rename{$(\mathcal{U},\mathcal{G})$, $\mathcal{G}=(\mathcal{G}_u)_{u\in \mathcal{U}}$} denotes the  Riemannian metric. 
	The induced norm is denoted by \rename{$\lVert \cdot \rVert_{\mathcal{G}} := \sqrt{\mathcal{G}(\cdot,\cdot)}$}. 
	\new{
		The tangent of space of $\mathcal{U}$ at a point $u \in \mathcal{U}$ is defined in its geometric version as 
		$$T_u\mathcal{U}=\{ c\colon \R \rightarrow \mathcal{U}\mid c\text{ differentiable}, c(0) =u\}/\sim,$$
		where the equivalence relation for two differentiable curves  $c,\tilde{c}\colon\R \rightarrow \mathcal{U}$  with $c(0) = \tilde{c}(0) =u$ is defined as follows:
		$c \sim \tilde{c} \Leftrightarrow \tfrac{\d}{\d t}\phi_{\alpha}(c(t))\vert_{t=0}  =\tfrac{\d}{\d t} \phi_{\alpha}(\tilde{c}(t))\vert_{t=0}$ for all $\alpha$ with $u \in U_\alpha$, where $\{(U_\alpha, \phi_\alpha)\}_\alpha \text{ is an  atlas of }\mathcal{U}.$
	}
	The derivative of a \new{smooth} mapping $f\colon \rename{\mathcal{U}}\rightarrow \rename{\widetilde{\mathcal{U}}}$ between two differentiable manifolds \rename{$\mathcal{U}$} and \rename{$\widetilde{\mathcal{U}}$} is defined using the pushforward.  In a point $u\in \rename{\mathcal{U}}$, it is defined by $(f_*)_u \colon T_u \rename{\mathcal{U}} \rightarrow T_{f(u)} \rename{\widetilde{\mathcal{U}}}$ with $(f_*)_u(c):= \frac{\mathrm{d}}{\mathrm{d}t} f(c(t))|_{t=0} = (f \circ c)'(0),$ where
	$c\colon I\subset \R \to \rename{\mathcal{U}}$ is a differentiable curve \new{with $c(0)=u$ and $c'(0)\in T_u \rename{\mathcal{U}}$}. In particular, $f\colon \rename{\mathcal{U}} \to \rename{\widetilde{\mathcal{U}}}$ is called $\mathcal{C}^k$ if $\psi_\beta \circ f\circ \phi_\alpha^{ -1}$ is $k$-times continuously differentiable for all charts $(U_\alpha,\phi_\alpha)$ of $\rename{\mathcal{U}}$ and $(V_\beta,\psi_\beta)$ of $\rename{\widetilde{\mathcal{U}}}$ with  $f(U_\alpha)\subset V_\beta$. 
	In the case $\rename{\widetilde{\mathcal{U}}}=\R$, a Riemannian gradient $\nabla f(u) \in T_u \rename{\mathcal{U}}$ is defined by the relation
	\begin{equation} 
		\label{DefGradient}
		(f_\ast)_u w
		= g_u(\nabla f(u), w) \quad \forall w \in T_u \rename{\mathcal{U}}.
	\end{equation}
	
	\new{
		We define $V_u\coloneqq\{v\in T_u\mathcal{U}\colon 1\in I_{u,v}^\mathcal{U}\}$ with $
		I_{u,v}^\mathcal{U}\coloneqq \bigcup\limits_{I\in\tilde{I}_{u,v}^\mathcal{U}} I $, where
		\begin{equation*}
			\begin{split}
				\tilde{I}_{u,v}^\mathcal{U}\coloneqq \{I\subset \R \colon & I \text{ open, }0\in I\text{, there exists a geodesic } c\colon I\to \mathcal{U} \\
				&\text{satisfying }c(0)=u\in \mathcal{U},\, c'(0)=v\in T_u\mathcal{U} \}.
			\end{split}
		\end{equation*}
		Then, we denote the exponential mapping by
		\begin{equation*}
			\exp\colon \bigcup\limits_{u\in \mathcal{U}}\{u\}\times V_u \to \mathcal{U},\ (u,v)\mapsto \exp_u(v)\coloneqq c(1),
		\end{equation*}
		where  $\exp_u(v)$ is the exponential map of $\mathcal{U}$ at $U$, which assigns to every tangent vector $v\in V_u $ the point $c(1)$ and $c\colon I_{u,v}^\mathcal{U}\to U$ is the unique geodesic satisfying $c(0)=u$ and $c'(0)=v$.}
	
	Let the length of a $\mathcal{C}^1$-curve $c\colon [0,1]\to \rename{\mathcal{U}}$ be denoted by $\textup{L}(c) = \int_0^1 \|c'(t)\|_g\d t$. 
	Then the distance $\mathrm{d}\colon \rename{\mathcal{U}} \times \rename{\mathcal{U}} \rightarrow \R$ between points $u,q\in \rename{\mathcal{U}}$  is given by
	\begin{align*}
		\mathrm{d}(u,q) 
		= \inf \{\textup{L}(c) \colon &  c\colon [0,1]\to \rename{\mathcal{U}}  \text{ is a piecewise smooth curve}\\ 
		&\text{with } c(0)=u \text{ and } c(1)=q\}.
	\end{align*}
	The injectivity radius $i_u$ at a point $u \in \rename{\mathcal{U}}$ is defined as $$i_u := \sup\{r > 0 \colon \exp_u \vert_{B_r(0_u)} \text{ is a diffeomorphism}\},$$where $0_u$ denotes the zero element of $T_u \rename{\mathcal{U}}$ and $B_r(0_u) \subset T_u \rename{\mathcal{U}}$ is a ball centered at $0_u\in T_u \rename{\mathcal{U}}$ with radius $r$. The injectivity radius of the manifold $\rename{\mathcal{U}}$ is the number $i(\rename{\mathcal{U}}) := \inf_{u \in \rename{\mathcal{U}}} i_u.$ 
	
	The triple $(\Omega, \mathcal{F}, \pP)$ denotes a (complete) probability space, where $\mathcal{F} \subset 2^{\Omega}$ is the $\sigma$-algebra of events and $\pP\colon \Omega \rightarrow [0,1]$ is a probability measure. The expectation of a random variable $X\colon \Omega \rightarrow \R$ is defined by $\E[X] = \int_\Omega X(\omega)\d \pP(\omega)$. A filtration is a sequence $\{ \mathcal{F}_n\}$ of sub-$\sigma$-algebras of $\mathcal{F}$ such that $\mathcal{F}_1 \subset \mathcal{F}_2 \subset \cdots \subset \mathcal{F}$. If for an event $F \in \mathcal{F}$ it holds that $\pP(F) = 1$, then we say $F$ occurs almost surely (a.s.).
	Given an integrable random variable $X \colon \Omega \rightarrow \R$ and a sub-$\sigma$-algebra $\mathcal{F}_n$, the conditional expectation is denoted by $\E[X | \mathcal{F}_n]$, which is a random variable that is $\mathcal{F}_n$-measurable and satisfies $\int_A \E[X|\mathcal{F}_n](\omega) \d \pP(\omega) = \int_A X(\omega) \d \pP(\omega)$ for all $A \in \mathcal{F}_n$.
	
	We will frequently use the convention $\vec{\xi} \in \Xi$ to denote a realization (i.e., the deterministic value $\vec{\xi}(\omega) \in \Xi$ for some $\omega$) of the vector $\vec{\xi}\colon \Omega \rightarrow \Xi \subset \R^m$; based on the context, there should be no confusion as to whether a realization or a random vector is meant.
	Let $J\colon \new{\mathcal{U}} \times \R^m \rightarrow \R$ be a parametrized objective as in problem~\eqref{eq:SO-problem-abstract} and define $J_{\vec{\xi}}:=J(\cdot,\vec{\xi})$. The gradient  $\nabla_{\new{u}} J(\new{u},\vec{\xi}):=\nabla J_{\vec{\xi}}(\new{u})$ of $J$ with respect to $\new{u}$ is defined by the relation
	\begin{equation}
		\label{eq:parametrized-gradient}
		((J_{\vec{\xi}})_\ast)_{\new{u}} \new{w}
		= \new{\mathcal{G}}_{\new{u}} (\nabla_{\new{u}} J(\new{u},\vec{\xi}),\new{w}) \quad \forall\new{w}\in T_{\new{u}} \new{\mathcal{U}}.
	\end{equation}
	Following~\cite{Geiersbach2021}, if $\nabla_{\new{u}} J\colon \new{\mathcal{U}} \times \R^m \rightarrow T_{\new{u}} \new{\mathcal{U}}$ is $\pP$-integrable, equation (\ref{eq:parametrized-gradient}) is fulfilled for all $\new{u}$ almost surely, and $\E[\nabla_{\new{u}} J(\new{u},\vec{\xi})] = \nabla j(\new{u})$, we call $\nabla_{\new{u}} J$ a stochastic gradient. 

	Let the Lagrangian for problem (\ref{eq:SO-problem-abstract}) be the mapping $\mathcal{L}\colon \new{\mathcal{U}} \times \R^n \rightarrow \R$ defined by
	\begin{equation*}
		\mathcal{L}(\new{u}, \vl):=  j(\new{u}) +  \vec{\lambda}^\top \vec{h}(\new{u}).
	\end{equation*} 
	The gradient $\nabla h_i(\new{u}) \in T_{\new{u}} \new{\mathcal{U}}$ of $h_i \colon \new{\mathcal{U}} \rightarrow \R$ is defined by the relation $((h_i)_{*})_{\new{u}} \new{w} = \new{\mathcal{G}}(\nabla h_i(\new{u}),\new{w})$ for all $\new{w} \in T_{\new{u}} \new{\mathcal{U}}.$
	The gradient of the corresponding vector $\vec{h}\colon \new{\mathcal{U}} \rightarrow \R^n$ is the vector $\nabla \vec{h} (\new{u})= (\nabla h_1(\new{u}), \dots, \nabla h_n(\new{u}))^\top.$
	
	In the following, we define a  \emph{Karush--Kuhn--Tucker} (KKT) point.

	\begin{definition}
		The pair $(\hat{\new{u}},\hat{\vec{\lambda}})\in \mathcal{U} \times \R^n$ is called a \textit{KKT point} for problem (\ref{eq:SO-problem-abstract}) if it satisfies the following conditions:
		\begin{subequations}
			\label{eq:KKT-conditions}
			\begin{align}
				\nabla j(\hat{\new{u}})+ \sum_{i=1}^n \hat{\lambda}_i \nabla h_i({\hat{\new{u}}})&= 0_{\hat{\new{u}}}, \label{eq:KKT-conditions-1}\\
				h_i(\hat{\new{u}}) &=0, \quad \forall i \in \mathcal{E},\label{eq:KKT-conditions-2}\\
				h_i(\hat{\new{u}}) \leq 0, \quad {\hat{\lambda}_i} \geq 0, \quad {{\hat{\lambda}}_i} h_i(\hat{\new{u}}) &= 0, \quad \forall i \in \mathcal{I}. \label{eq:KKT-conditions-3}
			\end{align}
		\end{subequations}
	\end{definition}
	
	\new{
		\begin{remark}
			In order for the above-formulated KKT conditions to be necessary optimality conditions for problem (\ref{eq:SO-problem-abstract}), \new{certain constraint qualifications are required}. 
			Analogues of linear independence (LICQ), Mangasarian-Fromovitz, Abadie, and Guignard constraint qualifications have only recently been treated in finite-dimensional manifolds; see \cite{bergmann2019intrinsic,yang2014}. The investigation of proper constraint qualifications for the infinite dimensional setting is an open area of research and is not further pursued in this paper. In Theorem~\ref{theorem:convergence-properties}, we will see that in certain cases, our method produces KKT points in the limit. However, in the absence of constraint qualifications, it can  only be shown that certain asymptotic KKT conditions (AKKT) are satisfied, in general. This is discussed in more detail in Section~\ref{subsec:outer-loop}.
		\end{remark}
	}
	
	The closed cone corresponding to the constraints, the  distance to the cone, and the projection are defined, respectively, by
	\begin{align*}
		\vec{K} & := \{ \vec{y} \in \R^n: y_i=0 \,\, \forall i \in \mathcal{E},  y_i \leq 0 \,\, \forall i \in \mathcal{I}\}, \\
		\operatorname{dist}_{\vec{K}}(\vec{y})&:=\inf_{\vec{k} \in {\vec{K}}} \lVert \vec{y}-\vec{k}\rVert_2 \quad \text{and} \quad 
		\pi_{\vec{K}}(\vec{y})  := \operatorname*{argmin}_{\vec{k} \in {\vec{K}}}\lVert \vec{y}-\vec{k}\rVert_2.
	\end{align*}
	For $y \in \R$, the projection \new{onto the $i$th component of the closed cone $\vec{K}$} has the formula $\pi_{K_i}(y) = 0$ if $i \in \mathcal{E}$, and $\pi_{K_i}(y) = \min(0,y)$ if $i \in \mathcal{I}$. We have $\pi_{\vec{K}}(\vec{y}) = (\pi_{K_1}(y_1), \dots, \pi_{K_n}(y_n))^\top.$ The normal cone of $\vK$ in a point $\vec{s} \in \vK$ is defined by $N_{\vK}(\vec{s})= \{ \vec{v} \in \R^n\colon  \vec{v}^\top (\vec{s}- \vec{y}) \geq 0 \,\, \forall \vec{y} \in \vK \}$; the normal cone is the empty set if $\vec{s}$ is not contained in $\vK$.
	To define the augmented Lagrangian, we first introduce a slack variable $\vec{s} \in {\vec{K}}$ to obtain the equivalent, equality-constrained problem
	\begin{equation*}
		\begin{aligned}
			&\min_{(\new{u},\vec{s}) \in \new{\mathcal{U}}\times {\vec{K}}} \, \{  j(\new{u})=\E[J(\new{u},\vec{\xi})] \} \quad \text{s.t.} \quad \vec{h}(\new{u})-\vec{s} = \vec{0}. 
		\end{aligned}
	\end{equation*}
	The corresponding augmented Lagrangian for a fixed parameter $\mu$ is the mapping  $\mathcal{L}_A^{\vec{s}}\colon \new{\mathcal{U}} \times \R^n \times \R^n \rightarrow \R$ defined by
	\begin{align*}
		\mathcal{L}_A^{\vec{s}}(\new{u}, \vec{s},\vec{\lambda};\mu) &= j(\new{u}) +  \vec{\lambda}^\top( \vec{h}(\new{u})-\vec{s}) + \frac{\mu}{2} \lVert \vec{h}(\new{u})-\vec{s}\rVert_2^2 \\& =j(\new{u}) + \frac{\mu}{2} \Big\lVert \vec{h}(\new{u})+ \frac{\vec{\lambda}}{\mu} - \vec{s} \Big\rVert_2^2 - \frac{\lVert \vec{\lambda}\rVert_2^2}{2\mu}.
	\end{align*}
	Notice that 
	$\min_{\vec{s} \in \vec{K}} \lVert \vec{h}(\new{u}) + \tfrac{\vec{\lambda}}{\mu} - \vec{s}\rVert_2^2 = \operatorname{dist}_{\vec{K}}(\vec{h}(\new{u})+\tfrac{\vec{\lambda}}{\mu})^2.$
	Hence, it is possible to eliminate the slack variable to obtain, again for fixed $\mu$, the augmented Lagrangian $\mathcal{L}_A\colon \new{\mathcal{U}} \times \R^n \rightarrow \R$ defined by
	
	\begin{equation*}
		\mathcal{L}_A(\new{u}, \vec{\lambda};\mu) = j(\new{u}) + \frac{\mu}{2}\operatorname{dist}_{\vec{K}}\left(\vec{h}(\new{u})+ \frac{\vec{\lambda}}{\mu}\right)^2 - \frac{\lVert \vec{\lambda}\rVert_2^2}{2\mu}.
	\end{equation*}

	\subsection{Augmented Lagrangian Method on Riemannian Manifolds}
	\label{subsec:algorithm}
	In this section, we present Algorithm \ref{alg:algorithmAugLagr}, which relies on stochastic approximation. For this, we need the function $L_A \colon \new{\mathcal{U}} \times \R^n \times \Xi \rightarrow \R$ defined by
	\begin{equation*}
		L_A(\new{u},\vec{\lambda}, \vec{\xi};\mu):= J(\new{u},\vec{\xi})+  \frac{\mu}{2}\operatorname{dist}_{\vec{K}}\left(\vec{h}(\new{u})+ \frac{\vec{\lambda}}{\mu}\right)^2 - \frac{\lVert \vec{\lambda}\rVert_2^2}{2\mu}.
	\end{equation*}

	The stochastic augmented Lagrangian (AL) method is shown in Algorithm~\ref{alg:algorithmAugLagr}. The inner loop is an adaptation of the randomized mini-batch stochastic gradient (RSG) method from \cite{Ghadimi2016}. In deterministic AL methods, the inner loop is in practice only solved up to a given error tolerance, leading to an \textit{inexact} augmented Lagrangian method. Deterministic termination conditions for the inner loop typically rely on conditions of the following type: $\new{u}^{k+1}$ is chosen as the first point of the corresponding iterative procedure satisfying 
	\begin{equation*}
		\nabla_{\new{u}} \mathcal{L}_A(\new{u}^{k+1},\vec{w}^k; \mu_k) =\varepsilon_k
	\end{equation*}
	with the error disappearing asymptotically, i.e., $\varepsilon_k \rightarrow 0$ as $k \rightarrow \infty$.
	Stochastic methods like the kind used here can only provide probabilistic error bounds; termination conditions are based on a priori estimates and result in stochastic errors.
	The outer loop corresponds to the augmented Lagrangian (AL) method with a safeguarding procedure as described in \cite{kanzow2018error}; see also \cite{Steck2018}. A feature of this procedure is that instead of using the Lagrange multiplier \new{$\vec{\lambda}$} in the subproblem in line 4, one chooses a function \new{$\vec{w}$} from a bounded set $B$, which is essential for achieving global convergence. In practice, this should be chosen in such a way so that the projection is easy to compute, i.e., box constraints are appropriate. A natural choice is $\vec{w}^k:=\pi_B(\vec{\lambda}^k)$ \new{for a closed and convex set $B$}. \new{For the algorithm, we define a infeasibility measure and its induced sequence by
		\begin{equation*}
			H(\new{u},\vec{\lambda}; \mu) := \left\lVert \vec{h}(\new{u}) - \pi_{\vec{K}}\left( \vec{h}(\new{u}) +\frac{\vec{\lambda}}{\mu} \right) \right\rVert_2 \quad \text{and} \quad H_{k}:=H(\new{u}^{k},\vec{w}^{k-1};\mu_{k-1}).
	\end{equation*}}

	\begin{algorithm}
		\caption{Stochastic Augmented Lagrangian Method}
		\label{alg:algorithmAugLagr}
		\begin{algorithmic}[1]
			\State \textbf{Input:} Initial point $\new{u}^1 \in \new{\mathcal{U}}$, 
			AL parameters $\gamma > 1$, $\tau \in (0,1),\, B \subset \R^n$ 
			\State \textbf{Initialization:} $\mu_1 > 0$, $\vec{\lambda}^1 \in \R^n$,  $k:=1$ \setlength\itemsep{0ex}
			\While{$\new{u}^k$, $\vec{\lambda}^k$ not converged}
			\State Choose $\vec{w}^k \in B$, step size $t_k$, iteration limit $N_k$, and batch size $m_k$
			\State $\new{z}^{k,1}:=\new{u}^k$
			\State Take a sample $R_k$ from the uniform distribution on $\{ 1, \dots, N_k\}$ 
			\For{$j=1, \dots, R_k$}
			\State Take i.i.d.~samples $\{\vec{\xi}^{k,j,1}, \dots, \vec{\xi}^{k,j,m_k}\} $ according to probability distribution $\pP$ 
			\State $\new{z}^{k,j+1}:=\new{\exp}_{\new{z}^{k,j}}(- \frac{t_k}{m_k}\sum_{s=1}^{m_k}\nabla_{\new{u}} L_A(\new{z}^{k,j},\vec{w}^k, \vec{\xi}^{k,j,s};\mu_k))$
			\EndFor
			\State $\new{u}^{k+1}:=\new{z}^{k,j+1}$
			\State  $\vec{\lambda}^{k+1}:=\mu_k \left(\vec{h}(\new{u}^{k+1}) + \frac{\vec{w}^k}{\mu_k}  - \pi_{\vec{K}}\left( \vec{h}(\new{u}^{k+1}) +  \frac{\vec{w}^k}{\mu_k} \right) \right)$
			\State If $H_{k+1} \leq \tau H_k$ or $k=1$ satisfied, set $\mu_{k+1} = \mu_k$. Otherwise, set $\mu_{k+1}:=\gamma \mu_k$.
			\State $k:=k+1$
			\EndWhile
		\end{algorithmic}
	\end{algorithm}
	
	\subsection{Convergence of Inner Loop}
	\label{subsec:convergence}
	To prove convergence of the RSG procedure in Algorithm~\ref{alg:algorithmAugLagr}, we make the following assumptions about the manifold, which are adapted from \cite{Geiersbach2021}.
	\begin{ass}
		\label{assump:manifold}
		We assume that (i) the distance $\mathrm{d}(\cdot,\cdot)$ is non-degenerate,\\
		(ii) the manifold $(\new{\mathcal{U}},\new{\mathcal{G}})$ has a positive injectivity radius $i(\new{\mathcal{U}})$, and\\
		(iii) for all $\new{u}\in\new{\mathcal{U}}$ and all $\tilde{\new{u}} \in \new{\exp_{\vec{u}}(B_{i_{\new{u}}}(0_{\new{u}}))}$, the minimizing geodesic between $\new{u}$ and $\tilde{\new{u}}$ is completely contained in $B_{i_{\new{u}}}(0_{\new{u}})$. 
\end{ass}

As pointed out in \cite{Geiersbach2021}, the conditions in Assumption~\ref{assump:manifold}, \new{while mild for finite-dimensional manifolds,} are strong for infinite-dimensional manifolds. \new{Distances on an infinite-dimensional Riemannian manifold can be degenerate. For example, \cite{MichorMumford1} shows that the reparametrization invariant $L^2$-metric on the infinite-dimensional manifold of smooth planar curves induces a geodesic distance equal to zero. Any assumption regarding the injectivity radius is challenging to prove in practice.} In infinite dimensions, Riemannian metrics are generally weak, so that gradients may not exist. \new{For certain metrics, the exponential map may not be well-defined; it may even fail to be a diffeomorphism on any neighborhood, see, e.g.,~\cite{constantin2007geodesic}.}

In the following, a function $g \colon \new{\mathcal{U}} \rightarrow \R$ is called $L_g$-Lipschitz continuously differentiable if the function is $\mathcal{C}^1$ and there exists a constant $L_g>0$ such that for all $\new{u}, \tilde{\new{u}} \in \new{\mathcal{U}}$ with $\mathrm{d}(\new{u},\tilde{\new{u}})< i(\new{\mathcal{U}})$ we have
\begin{align*}
	\lVert P_{1,0}\nabla j(\tilde{\new{u}})-\nabla j(\new{u})\rVert_{\new{\mathcal{G}}} &\leq L_j \mathrm{d}(\new{u},\tilde{\new{u}}),
\end{align*}
where $P_{1,0}\colon T_{\gamma(1)}\new{\mathcal{U}} \rightarrow T_{\gamma(0)}\new{\mathcal{U}}$ is the parallel transport along the unique geodesic such that $\gamma(0) = \new{u}$ and $\gamma(1) = \tilde{\new{u}}.$
\begin{ass} 
	\label{assump:functions}
	\begin{enumerate}
		\item[(i)] The functions $j$ and ${h}_i$ ($i=1, \dots, n$) are $L_j$-Lipschitz and $L_{{h}_i}$-Lipschitz continuously differentiable and the gradients $\nabla j$ and $\nabla {h}_i$ ($i=1, \dots, n$) exist for all $\new{u} \in \new{\mathcal{U}}$.
		\item[(ii)]  \new{The function $J$ is continuously differentiable with respect to the first argument for every $\vec{\xi} \in \Xi$, the} stochastic gradient $\nabla_{\new{u}} J$ defined by \eqref{eq:parametrized-gradient} exists, and there exists $M>0$ such that:
		\begin{equation}
			\label{eq:expectation-variance-condition}
			\E[\lVert \nabla_{\new{u}} J(\new{u},\vec{\xi} ) - \nabla j(\new{u})\rVert_{\new{\mathcal{G}}}^2] \leq M^2 \quad  \forall \new{u} \in \new{\mathcal{U}}.
		\end{equation}
	\end{enumerate}
\end{ass}

We begin our investigations with the following useful property.
\begin{lemma}
	\label{lem:lagrangian-identity}
	Under Assumption~\ref{assump:manifold} and assuming the gradients $\nabla j$ and $\nabla {h}_i$ ($i=1, \dots, n$) exist, the iterates of  Algorithm~\ref{alg:algorithmAugLagr} satisfy
	\[
	\nabla_{\new{u}} \mathcal{L}_A(\new{u}^{k+1}, \vw^k; \mu_k) = \nabla_{\new{u}} \mathcal{L}(\new{u}^{k+1}, \vl^{k+1}) \quad \text{for all $k$.}
	\]
\end{lemma}

\begin{proof}
	We have $\nabla \operatorname{dist}_{\vec{K}}^2 = 2(\textup{Id}_{\R^n} - \pi_{\vec{K}})$; see \cite[Corollary 12.31]{Bauschke2017}. Let $f(\new{u}):=\mathcal{L}_A(\new{u},\vec{w}; \mu)$. Then, the chain rule yields
	\begin{equation*}
		\begin{aligned}
			(f_*)_{\new{u}} \new{v} &= (j_*)_{\new{u}} {\new{v}} + \mu \sum_{i=1}^n \left ( h_i(\new{u}) + \frac{w_i}{\mu} - \pi_{K^i} \left( h_i(\new{u})+  \frac{w_i}{\mu}  \right) \right)  ((h_i)_*)_{\new{u}} \new{v}. \\
		\end{aligned}
	\end{equation*}
	From this, thanks to the identity (\ref{DefGradient}), \new{it follows} that
	\begin{equation*}
		\nabla f(\new{u}) = \nabla j(\new{u}) + \mu \nabla \vh(\new{u})^\top \left( \vh(\new{u}) + \frac{\vec{w}}{\mu} - \pi_{\vec{K}} \left( \vh(\new{u})+  \frac{\vec{w}}{\mu}  \right) \right),
	\end{equation*}
	and using the definition of $\vl^{k+1}$ from Algorithm~\ref{alg:algorithmAugLagr}, we obtain
	\begin{equation*}
		\nabla_{\new{u}} \mathcal{L}_A(\new{u}^{k+1}, \vw^k;\mu_k) = \nabla j(\new{u}^{k+1}) +  \nabla \vh({\new{u}^{k+1}})^\top \vl^{k+1}.
	\end{equation*}
	Using the fact that $\nabla_{\new{u}} \mathcal{L}(\new{u},\vl) = \nabla j(\new{u}) +  \nabla \vh({\new{u}})^\top \vl$, we have proven the claim.
\end{proof}

Now, we turn to an efficiency estimate for the inner loop. First, we define the functions
\begin{align*}
	&F_k(\new{u},\vec{\xi}):=L_A(\new{u},\vec{w}^k, \vec{\xi};\mu_k) \quad \text{and} \\
	&f_k(\new{u}) := \E[L_A(\new{u},\vec{w}^k, \vec{\xi};\mu_k)] = \mathcal{L}_A(\new{u},\vec{w}^k;\mu_k).
\end{align*}
Recall the convention $\vec{\xi} \in \Xi$ being used in the definition of $F_k$ and $\vec{\xi} \colon \Omega \rightarrow \Xi$ being used in the definition of $f_k$.
\begin{lemma}
	\label{lemma:descent-method}
	Suppose that Assumption~\ref{assump:manifold} and Assumption~\ref{assump:functions} are satisfied and let $\hat{B}_k \subset \new{\mathcal{U}}$ be a bounded set such that $\mathrm{d}(\tilde{\new{u}},\new{u}) \leq i(\new{\mathcal{U}})$  for all $\tilde{\new{u}}, {\new{u}} \in \hat{B}_k.$ 
	Then, $f_k$ is $L_k$-Lipschitz continuously differentiable with $L_k$ depending on $L_j, L_{{h}_1}, \dots, L_{{h}_n}$, and $\hat{B}_k$. Moreover, for all $\tilde{\new{u}},\new{u} \in \hat{B}_k$ with $\new{v}:=\exp_{\new{u}}^{-1}(\tilde{\new{u}})$, we have
	\begin{equation}
		\label{eq:LS-gradient-descent}
		f_k(\tilde{\new{u}})-f_k(\new{u}) \leq \new{\mathcal{G}}(\nabla f_k(\new{u}), \new{v}) + \frac{L_k}{2}\lVert \new{v}\rVert_{\new{\mathcal{G}}}^2.
	\end{equation}
\end{lemma}
\begin{proof}
	Let $P_{1,0}$ denote the parallel transport as defined directly before Assumption~\ref{assump:functions} and set $g_i({\new{u}}):=h_i({\new{u}}) + \frac{w^k_i}{\mu_k}-\pi_{K_i}(h_i({\new{u}}) + \frac{w^k_i}{\mu_k}).$ Since $h_i$ is $L_{h_i}$-Lipschitz continuously differentiable and $\hat{B}_k$ is bounded, there exists $C_{i,k}>0$ such that $\lVert \nabla h_i(\new{u})\rVert_{\new{\mathcal{G}}} \leq C_{i,k}.$ Now, we have
	\begin{equation}
		\begin{aligned}
			&\Big\lVert \sum_{i=1}^n P_{1,0}\nabla h_i(\tilde{\new{u}})  g_i(\tilde{\new{u}}) - \nabla h_i(\new{u})  g_i(\new{u}) \Big\rVert_{\new{\mathcal{G}}} \\
			& \quad \leq \sum_{i=1}^n \lVert P_{1,0} \nabla h_i(\tilde{\new{u}})-\nabla h_i(\new{u})\rVert_{\new{\mathcal{G}}}\vert g_i(\new{u})\vert  + \lVert  \nabla h_i(\new{u})\rVert_{\new{\mathcal{G}}}\vert g_i(\tilde{\new{u}})- g_i(\new{u}) \vert \\
			&\quad \leq  \sum_{i=1}^n L_{h_i} \textup{d}(\new{u},\tilde{\new{u}}) \vert g_i(\new{u})\vert + C_{i,k} \vert g_i(\tilde{\new{u}})- g_i(\new{u}) \vert \\
			&\quad \leq  \sum_{i=1}^n L_{h_i} \textup{d}(\new{u},\tilde{\new{u}}) \vert g_i(\new{u})\vert + 2 C_{i,k} \vert h_i(\tilde{\new{u}})- h_i(\new{u}) \vert,
			\label{eq:proof-LS-h}
		\end{aligned}
	\end{equation}
	where in the last step, we used the contraction property of the projection operator.
	Notice that
	\begin{equation}
		\label{eq:proof-LS-h-1}
		|h_i(\tilde{\new{u}})-h_i(\new{u})| \leq C'_i \mathrm{d}(\tilde{\new{u}},\new{u})
	\end{equation}
	for some $C'_i>0$ since $h_i$ is $\mathcal{C}^1$. Additionally, we have
	\begin{equation}
		\label{eq:proof-LS-h-2}
		|g_i(\new{u})| \leq \Big\vert h_i(\new{u})+\frac{w_i^k}{\mu_k}\Big\vert \quad (i \in \mathcal{E})
	\end{equation}
	and
	\begin{equation}
		\label{eq:proof-LS-h-3}
		|g_i(\new{u})| = \begin{cases} h_i(\new{u})+\frac{w_i^k}{\mu_k} &\text{if }  h_i(\new{u})+\frac{w_i^k}{\mu_k}\geq 0, \\
			0 & \text{else}
		\end{cases} \quad (i \in \mathcal{I}).
	\end{equation}
	Since $\hat{B}_k$ is bounded, \eqref{eq:proof-LS-h-2} and \eqref{eq:proof-LS-h-3} together imply that there exists $C_{i,k}'' >0$ such that $\vert g_i(\new{u})| \leq C_{i,k}''$.
	As a consequence of \eqref{eq:proof-LS-h} and \eqref{eq:proof-LS-h-1}, we have
	\begin{align*}
		&\Big\lVert \sum_{i=1}^n P_{1,0}\nabla h_i(\tilde{\new{u}}) g_i(\tilde{\new{u}}) - \nabla h_i(\new{u})  g_i(\new{u}) \Big\rVert_{\new{\mathcal{G}}} \leq \mathrm{d}(\new{u},\tilde{\new{u}})\sum_{i=1}^n L_{{h}_i} C_{i,k}''+ 2 C_{i,k}C_i' .
	\end{align*}
	Setting $\tilde{L}_{\vec{h},k}:= \sum_{i=1}^n L_{{h}_i} C_{i,k}''+ 2 C_{i,k}C_i'$, we have
	\begin{align*}
		& \lVert P_{1,0} \nabla f_k(\tilde{\new{u}}) - \nabla f_k(\new{u})\rVert_{\new{\mathcal{G}}}\\
		&\leq \lVert  P_{1,0} \nabla j(\tilde{\new{u}}) -\nabla j(\new{u}) \rVert_{\new{\mathcal{G}}}  + \mu_k \Big\lVert \sum_{i=1}^n P_{1,0}\nabla h_i(\tilde{\new{u}})  g_i(\tilde{\new{u}}) - \nabla h_i(\new{u})  g_i(\new{u}) \Big\rVert_{\new{\mathcal{G}}}\\
		&  \leq (L_j +\mu_k \tilde{L}_{\vec{h},k}) \mathrm{d}(\tilde{\new{u}},\new{u})
	\end{align*} 
	Therefore, $f_k$ is $L_k$-Lipschitz with $L_k:=L_j +\mu_k \tilde{L}_{\vec{h},k}.$ Applying \cite[Theorem 2.6]{Geiersbach2021}, we obtain \eqref{eq:LS-gradient-descent}. 
\end{proof}
\begin{remark}
	In the previous lemma, we introduced a bounded set $\hat{B}_k$. For the following results, we will need the existence of these sets containing the iterates almost surely within each $k$. Conditions ensuring boundedness can, e.g., be guaranteed by including constraints of the form $\new{u} \in C\subset \new{\mathcal{U}} $ for some bounded set $C$, or growth conditions on the gradient in combination with a regularizer; see \cite{geiersbach2023stochastic}. 
\end{remark}

Our first result concerning the convergence of Algorithm~\ref{alg:algorithmAugLagr} handles the efficiency of the inner loop process, which corresponds to a stochastic gradient method that is randomly stopped after $R_k$ iterations. We follow the arguments in \cite[Corollary 3]{Ghadimi2016}. It is possible to choose non-constant step sizes $t_{k_j}$; see \cite[Theorem 2]{Ghadimi2016}, but for the sake of clarity we observe step sizes that are constant in the inner loop here. 

To handle the analysis, we interpret $R_k$  as a realization of a stopping time $\tau_k \colon \Omega \rightarrow \{ 1, \dots, N_k\}$. Let $\vec{\xi}^{k,j}:=(\vec{\xi}^{k,j,1}, \dots,\vec{ \xi}^{k,j,m_k})$ be the batch associated with iteration $j$ for a given outer loop $k$ and let $\mathcal{F}_{k,n} = \sigma(\vec{\xi}^{\ell,i}: \ell \in \{1, \dots, k\}, i \in \{ 1, \dots, n \}) $ define the corresponding natural filtration. We define the filtration associated with the randomly stopped stochastic process by $\mathcal{F}^{\tau_k} =\{ \mathcal{F}_{\ell,n \wedge \tau_k}: \ell \in \{ 1, \dots, k\}, n \in \{ 1, \dots, N_k\} \}$.

\begin{theorem}
	\label{thm:RSG-method}
	Suppose Assumption~\ref{assump:manifold} and Assumption~\ref{assump:functions} are satisfied. Observe a fixed iteration $k$ from Algorithm~\ref{alg:algorithmAugLagr}. Suppose the iterates $\{ \vec{z}^{k,j}\}$ are a.s.~contained in a bounded set $\hat{B}_k \subset \new{\mathcal{U}}$, where $\mathrm{d}(\new{u},\tilde{\new{u}})\leq i(\new{\mathcal{U}})$ for all $\new{u},\tilde{\new{u}}\in \hat{B}_k$. Then, if the step \new{size $t_{k}$ satisfies} $t_{k} =\alpha_k/{L_k}$ for $\alpha_k \in (0,2)$ \new{and all $k$,} we have
	\begin{equation}
		\label{eq:efficiency}
		{\E}[\lVert \nabla f_k(\new{u}^{k+1})\rVert_{\new{\mathcal{G}}}^2 | \mathcal{F}^{\tau_k} ] \leq \frac{2L_k(f_k(\new{u}^k)-f_k^*)}{(2\alpha_k-\alpha_k^2)N_k}+ \frac{\alpha_k M^2}{(2-\alpha_k)m_k},
	\end{equation}
	where $f_k^*:=\inf_{u \in \hat{B}_k} f_k(\new{u})$.
	Moreover, if $\hat{B}_\infty:=\cup_{k=1}^\infty \hat{B}_k$ is bounded, $\mathrm{d}(\new{u}, \tilde{\new{u}}) \leq i(\new{\mathcal{U}})$ for all $\new{u}, \tilde{\new{u}} \in \hat{B}_\infty$, the maximum iterations $\{ N_k\}$ are chosen such that $N_k = \beta_k L_k$ for $\beta_k >0$, and
	\begin{equation}
		\label{eq:finite-sum}
		\sum_{k=1}^\infty \frac{1}{(2\alpha_k-\alpha_k^2)\beta_k}+ \frac{\alpha_k}{(2-\alpha_k)m_k}  < \infty,
	\end{equation}
	then we have $\lVert  \nabla f_k(\new{u}^{k+1})\rVert_{\new{\mathcal{G}}} \rightarrow 0$ a.s.~as $k\rightarrow \infty$.
\end{theorem}
\begin{proof}
	Let $k$ be fixed. We define $\delta^{j}:= \frac{1}{m_k}\sum_{i=1}^{m_k} \nabla_{\new{u}}F_k(\new{z}^{k,j},\vec{\xi}^{k,j,i})- \nabla f_k(\new{z}^{k,j}).$ With $\new{v}^j:=\exp_{\new{z}^{k,j}}^{-1}(\new{z}^{k,j+1}) = - \frac{1}{L_k m_k}\sum_{i=1}^{m_k} \nabla_{\new{u}}F_k(\new{z}^{k,j},\vec{\xi}^{k,j,i})$,  Lemma~\ref{lemma:descent-method} yields
	\begin{align*}
		&f_k({\new{z}}^{k,j+1})-f_k(\new{z}^{k,j}) \\
		&\leq - t_k \new{\mathcal{G}}\left(\nabla f_k(\new{z}^{k,j}), \frac{1}{m_k}\sum_{i=1}^{m_k} \nabla_{\new{u}}F_k(\new{z}^{k,j},\vec{\xi}^{k,j,i})\right) \\
		&  \qquad+ \frac{L_kt_k^2}{2}\left \lVert  \frac{1}{m_k}\sum_{i=1}^{m_k} \nabla_{\new{u}}F_k(\new{z}^{k,j},\vec{\xi}^{k,j,i}) \right\rVert_{\new{\mathcal{G}}}^2 \\
		& = -\frac{\alpha_k}{L_k} \lVert \nabla f_k(\new{z}^{k,j})\rVert_{\new{\mathcal{G}}}^2 -\frac{\alpha_k}{L_k} \new{\mathcal{G}}(\nabla f_k(\new{z}^{k,j}),\delta^j) \\
		&  \qquad + \frac{\alpha_k^2}{2 L_k} \left( \lVert \nabla f_k(\new{z}^{k,j})\rVert_{\new{\mathcal{G}}}^2 + 2\new{\mathcal{G}}(\nabla f_k(\new{z}^{k,j}),\delta^j) + \lVert \delta^j\rVert_{\new{\mathcal{G}}}^2 
		\right) \\
		& =  \left(-\frac{\alpha_k}{L_k}+\frac{\alpha_k^2}{2L_k}\right) \lVert \nabla f_k(\new{z}^{k,j})\rVert_{\new{\mathcal{G}}}^2 +
		\left(-\frac{\alpha_k}{L_k}+\frac{\alpha_k^2}{L_k}\right) \new{\mathcal{G}}(\nabla f_k(\new{z}^{k,j}),\delta^j) \\
		& \qquad +
		\frac{\alpha_k^2}{2 L_k} \lVert \delta^j\rVert_{\new{\mathcal{G}}}^2. 
	\end{align*}
	Taking the sum with respect to $j$ on both sides and rearranging, we obtain 
	\begin{equation}
		\label{eq:a-b}
		\begin{aligned}
			&\sum_{\ell=1}^{N_k} \lVert \nabla f_k(\new{z}^{k,\ell})\rVert_{\new{\mathcal{G}}}^2 \leq \frac{2L_k}{2\alpha_k-\alpha_k^2}( f_k(\new{z}^{k,1}) -f_k^*) \\
			&\qquad+ \frac{2(\alpha_k-1)}{2-\alpha_k} \sum_{\ell=1}^{N_k} \new{\mathcal{G}}(\nabla f_k(\new{z}^{k,\ell}),\delta^\ell) 
			+\frac{\alpha_k}{2-\alpha_k} \sum_{\ell=1}^{N_k}  \lVert \delta^\ell \rVert_{\new{\mathcal{G}}}^2
		\end{aligned}
	\end{equation}
	since $f_k^* \leq f_k(\new{z}^{k,N_k+1})$ and $0 < \alpha_k < 2$. Since $\nabla_{\new{u}}F_k$ is a stochastic gradient, we have $\E[\new{\mathcal{G}}(\nabla f_k(\new{z}^{k,j}),\delta^j)|\mathcal{F}_{k,j}] =\new{\mathcal{G}} (\nabla f_k(\new{z}^{k,j}),\E[\delta^j|\mathcal{F}_{k,j}])=0.$
	Notice that due to  \eqref{eq:expectation-variance-condition}, we have 
	\begin{equation}
		\label{eq:tower-property}
		\begin{aligned}
			& \E  \left[  \lVert  \nabla_{\new{u}}F_k(\new{z}^{k,j},\vec{\xi}^{k,j,i})- \nabla f_k(\new{z}^{k,j}) \rVert_{\new{\mathcal{G}}}^2 \vert  \mathcal{F}_{k,j}\right]
			\\ &= \E  \left[  \lVert  \nabla_{\new{u}}J(\new{z}^{k,j},\vec{\xi}^{k,j,i})- \nabla j(\new{z}^{k,j}) \rVert_{\new{\mathcal{G}}}^2 \vert  \mathcal{F}_{k,j}\right] \\
			&=  \E  \left[  \lVert  \nabla_{\new{u}}J(\new{z}^{k,j},\vec{\xi})- \nabla j(\new{z}^{k,j}) \rVert^2_{\new{\mathcal{G}}} \right] \leq M^2.
		\end{aligned}
	\end{equation}
	With \eqref{eq:tower-property}, we obtain
	\begin{equation}
		\label{eq:conditional-inequality}
		\begin{aligned}
			\E[ \lVert \delta^j\rVert_{\new{\mathcal{G}}}^2 |\mathcal{F}_{k,j}] 
			&= \frac{1}{m_k^2} \E  \left[ \Big \lVert  \sum_{i=1}^{m_k} \left(\nabla_{\new{u}}F(\new{z}^{k,j},\vec{\xi}^{k,j,i})- \nabla f_k(\new{z}^{k,j})\right) \Big\rVert_{\new{\mathcal{G}}}^2 \Big\vert  \mathcal{F}_{k,j}\right] \\
			&\leq  \frac{1}{m_k^2}  \sum_{i=1}^{m_k}  \E  \left[  \lVert  \nabla_{\new{u}}F(\new{z}^{k,j},\vec{\xi}^{k,j,i})- \nabla f_k(\new{z}^{k,j}) \rVert_{\new{\mathcal{G}}}^2 \vert  \mathcal{F}_{k,j}\right] \leq \frac{M^2}{m_k},
		\end{aligned}
	\end{equation}
	where we used Jensen's inequality, the linearity of the expectation, and  \eqref{eq:tower-property}.
	Taking the expectation on both sides of \eqref{eq:conditional-inequality}, using \eqref{eq:a-b}, and using the tower \new{rule, cf.~\cite[Proposition 1.1 (a), p. 471]{Gut2013}}, we get the inequality
	\begin{equation}
		\label{eq:inequality-conditioned}
		\sum_{\ell=1}^{N_k}\E[ \lVert \nabla f_k(\new{z}^{k,\ell})\rVert_{\new{\mathcal{G}}}^2] \leq  \frac{2 L_k( f_k(\new{z}^{k,1}) -f_k^*)}{2\alpha_k-\alpha_k^2}+ \frac{\alpha_k}{2-\alpha_k}\frac{M^2 N_k}{m_k}.
	\end{equation}
	Due to the law of total expectation, we have 
	\begin{align*}
		{\E}[\lVert \nabla f_k(\new{z}^{k,R_k})\rVert_{\new{\mathcal{G}}}^2|\mathcal{F}^{\tau_k} ] &= {\E}[\lVert \nabla f_k(\new{z}^{k,\tau_k})\rVert_{\new{\mathcal{G}}}^2|\mathcal{F}^{\tau_k} ] \\
		&= \sum_{\ell=1}^{N_k} \E[\lVert \nabla f(\new{z}^{k,\ell})\rVert_{\new{\mathcal{G}}}^2| \mathcal{F}_{k,\ell}] \pP\{\tau_k = \ell\} \\
		&= \frac{1}{N_k} \sum_{\ell=1}^{N_k} \E[\lVert \nabla f(\new{z}^{k,\ell})\rVert_{\new{\mathcal{G}}}^2] .
	\end{align*}
	Note that $f_k(\new{z}^{k,R_k}) = f_k(\new{u}^{k+1})$ and $f_k(\new{z}^{k,1})=f_k(\new{u}^k)$. Returning to \eqref{eq:inequality-conditioned}, we obtain 
	\begin{equation*}
		{\E}[\lVert \nabla f_k(\new{u}^{k+1})\rVert_{\new{\mathcal{G}}}^2| \mathcal{F}^{\tau_k} ] \leq  \frac{2 L_k( f_k(\new{u}^k) -f_k^*)}{(2\alpha_k-\alpha_k^2)N_k}+ \frac{\alpha_k M^2}{(2-\alpha_k)m_k},
	\end{equation*}
	so we have shown \eqref{eq:efficiency}. 
	
	Now, to prove almost sure convergence, we first observe that if all iterates are contained in $\hat{B}_\infty$, we have
	\begin{equation}
		\label{eq:uniform-boundedness-of-iterates}
		f_k(\new{u}^k) -  f_k^* \leq 2 \sup_{\new{u} \in \hat{B}_\infty} |f_k(\new{u})| \leq C
	\end{equation}
	for some $C>0$ due to the assumed smoothness of $f_k$ on $\new{\mathcal{U}}$. Taking the total expectation of \eqref{eq:efficiency}, Markov's inequality in combination with Jensen's inequality gives 
	\begin{align*}
		{\pP}\{ \lVert  \nabla f_k(\new{u}^{k+1})\rVert_{\new{\mathcal{G}}} \geq \varepsilon \} &\leq \varepsilon^{-2} {\E}[\lVert \nabla f_k(\new{u}^{k+1})\rVert^2_{\new{\mathcal{G}}} ] \\
		&\leq \varepsilon^{-2} \left(\frac{2L_kC}{(2\alpha_k-\alpha_k^2)N_k}+ \frac{\alpha_k M^2}{(2-\alpha_k)m_k}\right).
	\end{align*}
	Since $N_k=\beta_k L_k$ and \eqref{eq:finite-sum} holds, the infinite sum of the right-hand side is finite for every $\varepsilon>0$, implying the almost sure convergence of $\{  \lVert  \nabla f_k(\new{u}^{k+1})\rVert_{\new{\mathcal{G}}} \}$ to zero.
\end{proof}

For the choice $t_k = 1/L_k$ \new{and \eqref{eq:uniform-boundedness-of-iterates}}, the efficiency estimate \eqref{eq:efficiency} evidently simplifies to 
${\E}[\lVert \nabla f_k(\new{u}^{k+1})\rVert_{\new{\mathcal{G}}}^2 ] \leq \frac{2L_k\new{C}}{N_k}+ \frac{M^2}{m_k}.$ In the next section, we will investigate optimality of the solution in the limit as $k$ is taken to infinity. Since the Lipschitz constant $L_k$ has a potential to be unbounded due to the penalty term $\mu_k$, the maximal number of iterations $N_k$ needs to be balanced appropriately in this case. To obtain almost sure convergence, we required $N_k = \beta_k L_k$ for $\beta_k >0$. Alternatively, if it can be guaranteed that $L_k$ is bounded for all $k$ (for instance by bounding $\mu_k$), then one could (asymptotically) choose $t_k = \alpha_k/L$ with $L=\sup_{k} L_k$.
Regarding complexity, it is possible to establish the inner loop's complexity as argued in \cite[Section 4.2]{Ghadimi2016}. We define a $(\varepsilon_k, \eta_k)$-solution to the problem 
$\min_{\new{u} \in \new{\mathcal{U}}} \, \{ f_k(\new{u})= \E[F_k(\new{u},\vec{\xi})] \}
$
as the point $\hat{\new{u}}$ that satisfies ${\pP}\{\lVert \nabla f_k(\hat{\new{u}})\rVert_{\new{\mathcal{G}}}^2 \leq \varepsilon_k \} \geq 1 - \eta_k.$ Ignoring some constants, for the choice $t_k=1/L_k$, the complexity can be bounded by
$\mathcal{O} \left( (\eta_k \varepsilon_k)^{-1} + M^2 \eta_k^{-2} \varepsilon_k^{-2} \right).$

\subsection{Convergence of Outer Loop}
\label{subsec:outer-loop}
In the final part of this section, we analyze the behavior of the outer loop of Algorithm~\ref{alg:algorithmAugLagr} adapting arguments from \cite{Steck2018,kanzow2018augmented}. We define an optimality measure and its induced sequence by
\begin{equation*}
	r(\new{u},\vec{\lambda}) = \lVert \nabla_{\new{u}} \mathcal{L}(\new{u},\vec{\lambda}) \rVert_{\new{\cG}} + \lVert \vec{h}(\new{u}) - \pi_{\vec{K}}(\vec{h}(\new{u}) + \vec{\lambda}) \rVert_2, \quad r_k:=r(\new{u}^k,\vl^k)
\end{equation*}
and make the following assumptions on iterates induced by Algorithm~\ref{alg:algorithmAugLagr}.
\begin{ass}
	\label{assump:problem}
	We assume that
	\begin{enumerate}[label=(\roman*), leftmargin=2.5em]
		\item the sequence	 $\{ \new{u}^k\}$ is a.s.~contained in a bounded set $\hat{B}_\infty$ such that $\textup{d}(\new{u},\tilde{\new{u}}) \leq i(\new{\mathcal{U}})$ for all $\new{u},\tilde{\new{u}} \in \hat{B}_\infty$,
		\item $\lVert \nabla_{\new{u}}\mathcal{L}_A(\new{u}^{k+1},\vec{w}^k;\mu_k)\rVert_{\new{\mathcal{G}}}\rightarrow 0$ a.s.~as $k\rightarrow \infty$,
		\item $\{(\new{u}^k,\vec{\lambda}^k)\}$ converges a.s.~to the set of KKT points and
		\item for $k$ sufficiently large, we have $\vec{w}^k = \vec{\lambda}^k$.
	\end{enumerate}
\end{ass}

Note that Theorem~\ref{thm:RSG-method} implies Assumption~\ref{assump:problem}(ii). \new{Assumption~\ref{assump:problem}(iii) requires that every limit point of every realization of the sequence $\{u^k, \vec{\lambda}^k\}$ is a KKT point.}
In the absence of constraint qualifications, one can still work with asymptotic KKT (AKKT) conditions; under certain conditions, it can even be shown that they are necessary conditions (see, e.g., \cite[Theorem 5.3]{kanzow2018augmented}). We will say that a feasible point $\hat{\new{u}}$ satisfies the AKKT conditions if there exists a sequence $\{ \new{u}^k\}$ such that $\mathrm{d}(\new{u}^k,\hat{\new{u}})\rightarrow 0$ and a sequence $\{\vec{\lambda}^k \}$ contained in the dual cone $\vec{K}^{\oplus}:=\{ \vec{y} \in \R^n\colon \vec{y}^\top \vec{k} \geq 0 \, \forall \vec{k} \in \vec{K}\}$ such that 
\begin{equation}
	\label{eq:AKKT}
	\|\nabla j(\new{u}^k) + \nabla \vec{h}(\new{u}^k)^\top \vec{\lambda}^k\|_{\new{\mathcal{G}}} \rightarrow 0 \quad \text{ and } \quad \pi_{\vec{K}} (-\vec{h}(\new{u}^k))^\top \vec{\lambda}^k \rightarrow 0
\end{equation}
as $k \rightarrow \infty.$

A fundamental difference in the stochastic variant of the augmented Lagrangian method is that limit points, as limits of the stochastic process $(\new{u}^k, \vec{\lambda}^k)$, are random. In the following, we will consider a fixed limit point $(\hat{\new{u}},\hat{\vec{\lambda}})$ and the corresponding set of paths converging to it. This motivates the definition of the set
\begin{equation}
	\label{eq:probability-set}
	E_{\hat{\new{u}},\hat{\vec{\lambda}}}:=\{ \omega \new{\in \Omega}: (\new{u}^k(\omega), \vec{\lambda}^k(\omega)) \rightarrow (\hat{\new{u}}, \hat{\vec{\lambda}}) \quad \text{a.s.~on a subsequence} \}.
\end{equation}
\new{Note that here, and in the following analysis, $\omega$ represents an outcome of the random process $(\vec{\xi}^{1,1}, \dots, \vec{\xi}^{1,R_1}, \vec{\xi}^{2,1}, \dots,  \vec{\xi}^{2,R_2}, \dots)$ induced by sampling and random stopping.}

\begin{theorem}
	\label{theorem:convergence-properties}
	Suppose Assumption~\ref{assump:manifold}--Assumption~\ref{assump:problem}(i)-(ii) are satisfied. 
	Let $E:=\{ \omega \in \Omega: \mu_k(\omega) \text{ is a.s.~bounded}\}$. Then,  $\{ \vec{\lambda}^k(\omega)\}$ is a.s.~bounded on $E$ and any limit point $(\hat{\new{u}}, \hat{\vec{\lambda}})$ of $\{(\new{u}^k(\omega),\vec{\lambda}^k(\omega)): \omega \in E, k \in \N \}$ is a KKT point. On the set $\Omega \backslash E$, if a limit point $\hat{\new{u}}$ is feasible, then it is a AKKT point.
\end{theorem}

\begin{proof}
	We will make arguments in \new{two} parts, \new{where we distinguish between the case of bounded and unbounded $\mu_k$}.\\
	\textbf{\new{Case} 1: Bounded $\mu_k$}. We first show that the sequence $\{ \vec{\lambda}^k\}$ is a.s.~bounded. Let $\vec{v}^{k+1}:= \vec{h}(\new{u}^{k+1}) +\frac{\vw^k}{\mu_k}$ and $\vec{y}^{k+1} := \pi_{\vec{K}}(\vec{v}^{k+1})$. By definition of $\vec{\lambda}^{k}$, we have
	\begin{equation}
		\label{eq:expression-h}
		\vec{h}(\new{u}^{k+1}) = \frac{1}{\mu_k}(\vec{\lambda}^{k+1}-\vec{w}^k)+ \vec{y}^{k+1}.
	\end{equation}
	Now, observe that the boundedness of $\{ \mu_k\}$ on $E$ implies that there exists a maximal iterate $\bar{k}$ in Algorithm~\ref{alg:algorithmAugLagr} such that $H_{k+1}\leq \tau H_k \leq \tau M$ is satisfied for every $k \geq \bar{k}$ and some $M>0$. This $M$ exists since $\vec{h}$ is $\mathcal{C}^1$ and $\new{u}^k$, $\vec{w}^k$, and $\mu_k$ are all bounded by assumption. In particular, $H_k \rightarrow 0$ as $k\rightarrow \infty$ on $E$. 
	In turn, \eqref{eq:expression-h} combined with the definition of $H_k$ implies the a.s.~convergence of $\| \vec{\lambda}^{k+1}-\vec{w}^k\|_2/\mu_k$ to zero, in turn implying $\| \vec{\lambda}^{k+1}-\vec{w}^k\|_2 \rightarrow 0$ for $k\rightarrow 0$. The boundedness of $\vec{w}^k$ guaranteed by Algorithm~\ref{alg:algorithmAugLagr} means therefore that $\{ \vec{\lambda}^k\}$ is bounded on $E$.
	
	Now, we prove that for any $\vec{y} \in \vec{K}$, there exists a nonnegative sequence $\gamma_k$ converging to zero and such that 
	\begin{equation}
		\label{eq:VI-nullsequence}
		(\vec{y}-\vec{h}(\new{u}^k))^\top\vec{\lambda}^k  \leq \gamma_k, \quad \omega \in E, k \in \N.
	\end{equation}
	With \cite[Theorem 3.14]{Bauschke2017}, the projection formula 
	\[(\vec{v}^{k+1}-\vec{y}^{k+1})^\top (\vec{y}^{k+1}-\vec{y}) \geq 0 \]
	holds for all $\vec{y} \in {\vec{K}}$, implying that $\vl^{k+1}=\mu_{k+1}(\vec{v}^{k+1}-\vec{y}^{k+1}) \in N_{\vK}(\vec{y}^{k+1}).$  
	Now, using $\vec{\lambda}^{k+1} \in N_{\vec{K}}(\vec{y}^{k+1})$ and \eqref{eq:expression-h}, we have
	\begin{align*}
		(\vec{y}-\vec{h}(\new{u}^{k+1}))^\top\vec{\lambda}^{k+1} &= \left(\vec{y}- \frac{1}{\mu_k}(\vec{\lambda}^{k+1}-\vec{w}^k)-\vec{y}^{k+1}\right)^\top\vec{\lambda}^{k+1}\\
		&\leq \frac{1}{\mu_k}((\vec{w}^k)^\top \vec{\lambda}^{k+1} - \lVert \vec{\lambda}^{k+1} \rVert_2^2)\\
		&= (\vec{y}^{k+1}-\vec{h}(\new{u}^{k+1}))^\top \vec{\lambda}^{k+1}=:\gamma_{k+1}.
	\end{align*}
	We have shown \eqref{eq:VI-nullsequence}. That $\{\gamma_k\}$ is a.s.~a null sequence follows from the fact that $\| \vec{\lambda}^{k+1}-\vec{w}^k\|_2/\mu_k$ a.s. converges to zero.
	
	Consider a subsequence of $\{ (\new{u}^k(\omega),\vec{\lambda}^k(\omega))\}$ that converge to a limit point $(\hat{\new{u}},\hat{\vec{\lambda}})$ for a fixed $\omega \in E_{\hat{\new{u}}, \hat{\vec{\lambda}}}$. We will prove that the limit point satisfies the KKT conditions
	\eqref{eq:KKT-conditions}. Continuity of $\nabla_{\new{u}} \mathcal{L}$ gives  $\lim_{k \rightarrow \infty} \nabla_{\new{u}} \mathcal{L}(\new{u}^k(\omega),\vec{\lambda}^k(\omega)) = \nabla_{\new{u}}\mathcal{L}(\hat{\new{u}}, \hat{\vec{\lambda}})$ and $\| \nabla_{\new{u}} \mathcal{L}(\hat{\new{u}}, \hat{\vec{\lambda}}) \|_{\new{\mathcal{G}}} = 0$ due to Assumption~\ref{assump:problem}(ii). By definition, $\nabla_{\new{u}} \mathcal{L}(\hat{\new{u}},\hat{\vec{\lambda}}) \in T_{\hat{\new{u}}}\new{\mathcal{U}}$, and the only element in $T_{\hat{\new{u}}}\new{\mathcal{U}}$ having norm zero is $0_{\hat{\new{u}}}$, thus \eqref{eq:KKT-conditions-1} is fulfilled. Since $\alpha_k \rightarrow 0$ a.s., we have that $(\vec{y}-\vec{h}(\hat{\new{u}}))^\top  \hat{\vec{\lambda}} \leq 0$ for all $\vec{y} \in \vec{K},$ implying that $\hat{\vec{\lambda}} \in N_{\vec{K}}(\vec{h}(\hat{\new{u}}))$. This immediately implies \eqref{eq:KKT-conditions-2}--\eqref{eq:KKT-conditions-3}.\\
	\noindent
	\textbf{\new{Case} 2: Unbounded $\mu_k$}. 
	Consider a fixed $\omega \in \Omega\backslash E$ and a sequence $\{ \new{u}^k(\omega)\}$ such that (possibly on a subsequence that we do not relabel) $\mathrm{d}(\new{u}^k(\omega), \hat{\new{u}}) \rightarrow 0$ as $k\rightarrow \infty$. Assumption~\ref{assump:problem}(ii) gives the first AKKT condition in \eqref{eq:AKKT}.
	It remains to prove that  $\pi_{\vec{K}} (-\vec{h}(\new{u}^k(\omega)))^\top \vec{\lambda}^k(\omega) \rightarrow 0$. Now, we define 
	\[\vec{p}^k(\omega):=(\mu_k(\omega) \vec{h}(\new{u}^{k+1}(\omega))+\vec{w}^k(\omega))^\top \pi_{\vec{K}}(-\vec{h}(\new{u}^{k+1}(\omega))).\]
	For readability, we will suppress the dependence on $\omega$.
	Since 
	\[\vec{\lambda}^{k+1} =  \mu_k \left( \vec{h}(\new{u}^{k+1})+\frac{\vec{w}^k}{\mu_k} - \pi_{\vec{K}}\left(\vec{h}(\new{u}^{k+1})+\frac{\vec{w}^k}{\mu_k} \right)\right)\]
	it is evidently enough to prove $\vec{p}^k \rightarrow 0$, since due to the contraction property of the projection, we have $\pi_{\vec{K}}(\vec{a}^k)^\top \vec{b}^k \rightarrow 0$ implies $\pi_{\vec{K}}(\vec{a}^k)^\top \pi_{\vec{K}}(\vec{b}^k) \rightarrow 0$ for any $\vec{a}^k, \vec{b}^k \in \R^n.$ Note that at least on a subsequence, we have $\vec{h}(\new{u}^{k+1}) \rightarrow\vec{h}(\hat{\new{u}})$ and $| \vec{h}(\new{u}^{k})|$ is bounded. 
	
	Consider first the case that $h_i(\hat{\new{u}})<0$. Then $\vec{h}(\new{u}^{k+1})\rightarrow \vec{h}(\hat{\new{u}})$ implies that $w_i^k+\mu_k h_i(\new{u}^{k+1})<0$ for $k$ sufficiently large, implying $\vec{p}^k \rightarrow 0$. 
	
	Consider now the case that $h_i(\hat{\new{u}})=0$. For a fixed $k$, if $h_i(\new{u}^{k+1}) \geq 0$ then $\vec{p}^k=0$. Else if $h_i(\new{u}^{k+1})<0$, then $p_i^k = (\mu_k h_i(\new{u}^{k+1}) +w_i^k) \pi_{\vec{K}}(-h_i(\new{u}^{k+1})) 
	\leq w_i^k |h_i(\new{u}^{k+1})|$. If $h_i(\new{u}^{k+1})<0$ infinitely many times, then $w_i^k |h_i(\new{u}^{k+1})| \rightarrow 0$, meaning  $\vec{p}^k \rightarrow 0$.
	
	Since $\vec{p}^k$ in both cases converges to zero and $\omega \in \Omega \backslash E$ was arbitrary, we have proven the claim.
\end{proof}

We now turn to local convergence statements. 
In the spirit of a local argument, we restrict our investigations to the study around a limit point for \textit{only those realizations converging to it}.  Again, we consider the set $E_{\hat{\new{u}},\hat{\vec{\lambda}}}$ defined in \eqref{eq:probability-set}.

\begin{lemma}
	\label{lem:technical-result-residual}
	Suppose Assumptions~\ref{assump:manifold}--\ref{assump:problem} hold. Let $(\hat{\new{u}}, \hat{\vl})$ be a limit point satisfying for some $c_1, c_2 >0$
	\begin{equation}
		\label{errorbound}
		c_1 r(\new{u},\vl) \leq \d(\new{u},\hat{\new{u}}) + \lVert \vl - \hat{\vl}\rVert_2 \leq c_2 r(\new{u},\vl)
	\end{equation}
	for all $(\new{u}, \vl)$ with $\new{u}$ near $\hat{\new{u}}$ and $r(\new{u},\vl)$ sufficiently small. Then we have for sufficiently large $k$ 
	\begin{equation*}
		\left(1-\frac{c_2}{\mu_k}  \right) r_{k+1} \leq \lVert \nabla_{\new{u}} \mathcal{L}_A(\new{u}^{k+1}, \vec{w}^k; \mu_k)\rVert_{\new{\cG}} + \frac{c_2}{\mu_k} r_k \quad \text{a.s. on } E_{\hat{\new{u}},\hat{\vec{\lambda}}}.
	\end{equation*}
\end{lemma}

\begin{proof}
	We have using Lemma \ref{lem:lagrangian-identity} and $\vec{w}^k = \vec{\lambda}^k$ that 
	\begin{equation}
		\label{eq:residual-iterate}
		r_{k+1} = \lVert \nabla_{\new{u}} \mathcal{L}_A(\new{u}^{k+1},\vl^{k};\mu_k)\rVert_{\new{\cG}} + \lVert \vh(\new{u}^{k+1}) - \pi_{\vec{K}}(\vh(\new{u}^{k+1})+\vl^{k+1})\rVert_2.
	\end{equation} 
	Let $\vec{v}^{k+1}:= \vec{h}(\new{u}^{k+1}) +\frac{\vw^k}{\mu_k}$ and $\vec{y}^{k+1} := \pi_{\vec{K}}(\vec{v}^{k+1})$. Then it follows that 
	\begin{equation*}
		\lVert \vec{y}^{k+1} - \pi_{\vK}(\vec{y}^{k+1}+ \vl^{k+1})\rVert_2 = 0 
	\end{equation*}
	since $\vl^{k+1} \in N_{\vK}(\vec{y}^{k+1})$ as argued in Part 1 of the proof of Theorem~\ref{theorem:convergence-properties}. Note that $\textup{Id}_{\R^n} - \pi_{\vK}$ is (firmly) nonexpansive (cf.~\cite[Prop.~12.27]{Bauschke2017}). It is an easy exercise to deduce that the mapping $\vec{y} \mapsto \vec{y} - \pi_{\vec{K}}(\vec{y}+\vl^{k+1})$ is nonexpansive as well, from which we can conclude 
	\begin{equation*}
		\label{eq:firm-nonexpansive}
		\begin{aligned}
			&\Big |\lVert \vec{h}(\new{u}^{k+1})- \pi_{\vK}(\vec{h}(\new{u}^{k+1})+\vl^{k+1})\rVert_2 - \lVert \vec{y}^{k+1} - \pi_{\vK}(\vec{y}^{k+1} + \vl^{k+1})\rVert_2 \Big| \\
			&\quad \leq
			\lVert \vec{h}(\new{u}^{k+1})- \pi_{\vK}(\vec{h}(\new{u}^{k+1})+\vl^{k+1})-\vec{y}^{k+1} + \pi_{\vK}(\vec{y}^{k+1} + \vl^{k+1})\rVert_2 \\
			& \quad \leq \lVert \vec{h}(\new{u}^{k+1}) - \vec{y}^{k+1}\rVert_2.
		\end{aligned}
	\end{equation*}
	Using the definition of $\vec{y}^{k+1}$ and $\vec{w}^k=\vec{\lambda}^k$, notice that 
	\begin{equation*}
		\begin{aligned}
			&\lVert \vec{h}(\new{u}^{k+1})- \pi_{\vK}(\vec{h}(\new{u}^{k+1})+\vl^{k+1})\rVert_2 \\
			&\leq  \lVert \vh(\new{u}^{k+1}) - \pi_{\vec{K}}(\vh(\new{u}^{k+1})+\vl^{k}/\mu_k)\rVert_2\\
			&=\frac{1}{\mu_k} \lVert \mu_k \vh(\new{u}^{k+1}) +\vl^k - \mu_k \pi_{\vec{K}}(\vh(\new{u}^{k+1})+\vl^{k}/\mu_k) -\vl^k \rVert_2\\
			& = \frac{1}{\mu_k} \lVert \vl^{k+1} - \vl^k\rVert_2.
		\end{aligned}
	\end{equation*}
	Returning to \eqref{eq:residual-iterate}, we obtain 
	\begin{equation}
		\label{eq:residual-iterate-bound}
		r_{k+1} \leq  \lVert \nabla_{\new{u}} \mathcal{L}_A(\new{u}^{k+1},\vl^{k};\mu_k)\rVert_{\new{\cG}} + \frac{1}{\mu_k} \left( \lVert \vl^{k+1} - \hat{\vl}\rVert_2 +  \lVert  \vl^k - \hat{\vl} \rVert_2\right).
	\end{equation}
	Since $\lim_{k\rightarrow \infty} \mathrm{d}(\new{u}^k, \hat{\new{u}}) = 0$ a.s.~on $E_{\hat{\new{u}}, \hat{\vec{\lambda}}}$, then for any  $\varepsilon>0$ there exists $\bar{k}$ such that $\mathrm{d}(\new{u}^k, \hat{\new{u}}) < \varepsilon$ for all $k\geq \bar{k}$ a.s. Possibly choosing $\bar{k}$ even larger, Assumption~\ref{assump:problem} combined with the positive injectivity radius further implies 
	$\lVert \vl^k(\omega) - \hat{\vl}\rVert_2 \leq c_2 r_k$ for almost all $\omega \in E_{\hat{\new{u}}, \hat{\vec{\lambda}}}$. 
	Using \eqref{eq:residual-iterate-bound}, we conclude that for almost all $\omega \in E_{\hat{\new{u}}, \hat{\vec{\lambda}}}$,
	\begin{equation*}
		r_{k+1} \leq \lVert \nabla_{\new{u}} \mathcal{L}_A(\new{u}^{k+1}(\omega),\vl^{k}(\omega);\mu_k(\omega))\rVert_{\new{\cG}}  + \frac{1}{\mu_k} (c_2 r_{k+1} + c_2 r_k),
	\end{equation*}
	for $k$ large enough. Rearranging terms proves the claim. 
\end{proof}

We are now ready to show the local rate of convergence.
\new{
	We recall the definition of convergence for the convenience of the reader: A sequence 
	$\{r_k\}$ that converges to $r^\ast$ is said to have order of convergence 
	$s \geq 1$ and rate of convergence $q$ if
	$${\displaystyle \lim _{k\rightarrow \infty }{\frac {\left|r_{k+1}-r^{*}\right|}{\left|r_{k}-r^{*}\right|^{s}}}=q .}$$
	Linear convergence occurs in the case $s=1$ and $q\in (0,1)$. Moreover, superlinear convergence occurs in all cases where $q>1$ and the case where $ s=1$ and $q=0$.
}
\begin{theorem}
	\label{thm:convergence-rate}
	Under the same assumptions as Lemma~\ref{lem:technical-result-residual}, assume further that $\lVert  \nabla_{\vu} \mathcal{L}_A(\vu^{k+1},\vl^{k};\mu_k)\rVert_{\cG^N} = o(r_k).$ 
	Then 
	\begin{enumerate}[label=\arabic*), leftmargin=1.5em]
		\item[1)] \new{Given the existence of} $\hat{\mu}_q >0$ such that if $\mu_k \geq \hat{\mu}_q$ for $k$ sufficiently large, 
		$\{(\vu^k, \vl^k)\}$ converges \new{linearly} to $(\hat{\vu}, \hat{\vl})$ a.s.~on $E_{\hat{\vec{u}},\hat{\vec{\lambda}}}$ \new{with convergence rate $q\in (0,1)$}.
		\item[2)] If $\mu_k \rightarrow \infty$, then 
		$(\vu^k, \vl^k) \rightarrow (\hat{\vu}, \hat{\vl})$  a.s.~on $E_{\hat{\vec{u}},\hat{\vec{\lambda}}}$ at a superlinear rate.
	\end{enumerate}
\end{theorem}
\begin{proof}
	Note that for $k$ large enough, we have $\vec{w}^k = \vl^k$ and Lemma \ref{lem:technical-result-residual} gives
	$$\left(1-\frac{c_2}{\mu_k}  \right) r_{k+1} \leq \lVert \nabla_{\vu} \mathcal{L}_A(\vu^{k+1}, \vec{w}^k; \mu_k)\rVert_{\cG^N} + \frac{c_2}{\mu_k} r_k = o(r_k) + \frac{c_2}{\mu_k} r_k.$$
	Taking $\mu_k$  \new{such that $\mu_k-c_2>0 $} gives
	$r_{k+1} \leq \frac{\mu_k}{\mu_k - c_2} \left( o(r_k) + \frac{c_2}{\mu_k} r_k \right) .$ 
	This implies 
	\begin{equation*}
		\frac{r_{k+1}}{r_k} \leq \frac{\mu_k}{\mu_k - c_2} \left( o(1) + \frac{c_2}{\mu_k}  \right)=\frac{c_2}{\mu_k - c_2} + o(1).
	\end{equation*}
	Thanks to the error bound \eqref{errorbound}, we get the corresponding rates for $\{(\vu^k, \vl^k)\}$. 
\end{proof}
In practice, the assumption $\lVert  \nabla_{\vu} \mathcal{L}_A(\vu^{k+1},\vl^{k};\mu_k)\rVert_{\cG^N} = o(r_k)$ is difficult to implement since one can only work with estimates $$\hat{f}_k \approx \E[L_A(\vu^{k+1}, \vec{\lambda}^k, \vec{\xi}; \mu_k)] = \mathcal{L}_A(\vu^{k+1}, \vec{\lambda}^k; \mu_k).$$ However, we have a convergence rate guaranteed in expectation by \eqref{eq:efficiency}, which can be used to choose appropriate sequences for $N_k$ and $m_k$. A possible heuristic is shown in the following section.

\section{Application and Numerical Results}
\label{sec:numericalResults} 

In this section, we present an application to a two-dimensional fluid-mechanical problem to demonstrate the algorithm. We denote the hold-all domain as $D=D(\vu)$, which is partitioned into $N+1$ disjoint subdomains $D_1, \ldots, D_{N+1}$, where $D_{N+1}$ represents the subdomain in which fluid is allowed to flow, and the other sets are obstacles around which the fluid is supposed to flow. The subdomain boundaries are defined as $\partial D_1 = u_1$, $\ldots$, $\partial D_N = u_N$, and $\partial D_{N+1} = \Gamma \cup u_1 \cup \cdots \cup u_N $, where $\Gamma$ is the outer boundary that is fixed and split into two disjoint parts
$	\GD$ and $\GN $
representing the Dirichlet and Neumann boundary, respectively.

\new{In  \cite{LoayzaGeiersbachWelkerHandbook}, a shape is seen as a point on an abstract manifold so that a collection of shapes can be viewed as a vector of points $\vec{u}=(u_1, \dots, u_N)$ in a product manifold $\mathcal{U}^N = \mathcal{U}_1 \times \cdots \times \mathcal{U}_N$, \new{where $\mathcal{U}_i$ are Riemannian manifolds for all $i=1,\dots,N$}. 
	In the following, our shapes are the above-mentioned obstacles leading to a multi-shape optimization problem.
	One should take into account that a product manifold is a manifold and, thus, all theoretical findings from the Section \ref{sec:optimization-approach} can also be applied to product manifolds.
}
We will work with a (possibly infinite-dimensional) connected Riemannian product manifold  $(\new{\cU},\new{\cG})=(\cU^N, \cG^N)$.  
As described in  \cite{LoayzaGeiersbachWelkerHandbook}, the tangent space
$T\mathcal{U}^N$ can be identified with the product of tangent spaces $T\mathcal{U}_1\times\dots\times T\mathcal{U}_N$ via 
$
T_{\vu} \cU^N \cong T_{u_1} \cU_1 \times \dots \times T_{u_N} \cU_N.
$
Additionally, the product metric $\mathcal{G}^N$ to the corresponding product shape space $\mathcal{U}^N$ can be defined via
$\mathcal{G}^N=\new{(\mathcal{G}^N_{\vu})_{\vu\in\mathcal{U}^N}}$,
where
\begin{equation}
	\label{eq:product_metric}
	\mathcal{G}^N_{\vu}(\vec{v},\vec{w}) =\sum_{i=1}^{N} \mathcal{G}_{\pi_i(u)}^{i}(\pi_{i_\ast}\vec{v},\pi_{i_\ast}\vec{w})\qquad\forall \,  
	\vec{v},\vec{w} \in T_u \cU^N
\end{equation}
and $\pi_i\colon \cU^N\to \cU_i$, $i=1, \dots, N$, correspond to canonical projections.
\new{If we work with multiple shapes $\vu$, the exponential map in Algorithm~\ref{alg:algorithmAugLagr} needs to be replaced by the so-called
	multi-exponential map. 
	Let $V_{\vec{u}}^N\coloneqq V_{u_1}\times \cdots \times V_{u_N}$, where   $V_{u_i}\coloneqq \{v_i\in T_{u_i}\mathcal{U}_i\colon 1\in I_{u_i,v_i}^{\mathcal{U}_i}\}$ for all $i=1,\dots,N$.
	Then, we define the multi-exponential map by
	$\exp_{\vec{u}}^N\colon V_{\vec{u}}^N\to \mathcal{U}^N,\, \vec{v}=(v_1,\dots,v_N)\mapsto (\exp_{u_1}v_1,\dots,\exp_{u_N}v_N)$
	for the vector $\vec{u}=(u_1,\dots,u_N)$, where 
	$\exp_{u_i}\colon \new{V_{u_i}}\to \mathcal{U}_i,\,v_i\mapsto \exp_{u_i}(v_i)$ for all $i=1,\dots,N$.}

The shape space we consider in the numerical experiments is the product space of plane unparametrized curves, i.e., 
$\mathcal{U}^N=B_e^N(S^1,\R^2)$.
The shape space $B_e(S^1,\R^2)$ is defined as the orbit space of $\mathrm{Emb}(S^1,\mathbb{R}^2)$ under the action by composition from the right by the Lie group $\mathrm{Diff}(S^1)$, i.e., $B_e(S^1,\R^2)	:=\text{Emb}(S^1,\mathbb{R}^2) / \text{Diff}(S^1)$   (cf., e.g., \cite{MichorMumford1}). 
Here, $\mathrm{Emb}(S^1,\R^2)$ denotes the set of all embeddings from the unit circle $S^1$ into $\R^2$, and $\mathrm{Diff}(S^1)$ is the set of all diffeomorphisms from $S^1$ into itself.
In \cite{KrieglMichor}, it is proven that the shape space $B_e(S^1,\R^2)$ is a smooth manifold; together with appropriate inner products, it is even a Riemannian manifold.
In our numerical experiments, we choose the Steklov--Poincar\'e metric defined in \cite{Schulz2016a}.
Originally, it is defined as a mapping from Sobolev spaces. To define a metric on $B_e(S^1,\R^2)$, the Steklov--Poincar\'e metric is restricted to  a mapping  from the tangent spaces, i.e., $ T_u B_e(S^1,\R^2) \times T_u B_e(S^1,\R^2) \rightarrow \R$, where $T_uB_e(S^1,\mathbb{R}^{2})\cong\left\{h\colon h=\alpha \vec{n},\, \alpha\in \mathcal{C}^\infty(S^{1})\right\}$.
Of course, one can choose a different metric on the shape space to represent the shape gradient. We focus on the Steklov--Poincar\'{e} metric due to its advantages in combination with the computational mesh (cf.~\cite{Siebenborn2017,Schulz2016a}).

The physical system on $D$ is described by the Stokes equations under uncertainty. Note that here, flow is modeled on the domain $D$ instead of $D_{N+1}$. This is done (in view of the tracking-type functional) to produce a shape derivative on the entire domain. Let $V(D) = \left\{ \velocity \in H^1( D, \R^2)\colon \velocity\vert_{\GD \cup \vec{u}} = \vec{0} \right\}$ denote the function space associated to the velocity for a fixed domain $D$. We neglect volume forces and consider a deterministic viscosity of the fluid. Inflow $\vec{g}$ on parts of the Dirichlet boundary is assumed to be uncertain and is modeled as a random field $\vec{g}\colon D \times \Xi  \rightarrow \R^2$ with regularity $\vec{g} \in L_{\pP}^2(\Xi, H^1( D, \R^2))$ and depending on $\vec{\xi} \colon \Omega \rightarrow \Xi \subset \R^m$. We will use the abbreviation $\vec{g}_{\vec{\xi}} = \vec{g}(\cdot,\vec{\xi})$. For each realization $\vec{\xi}$, consider Stokes flow in weak form: find $\velocity_{\vec{\xi}} \in H^1( D, \R^2)$ and $\pressure_{\vec{\xi}} \in L^2( D)$ such that $\velocity_{\vec{\xi}}-\vec{g}_{\vec{\xi}} \in V(D)$ and
\begin{subequations}
	\label{eq:weak-formulation}
	\begin{align}
		\intD{ \nabla \velocity_{\vec{\xi}} : \nabla \adjvelocity- \pressure_{\vec{\xi}} \Div{\adjvelocity}} &= 0 \quad \forall \adjvelocity \in V(D), \\
		\intD{ \adjpressure \Div{\velocity}_{\vec{\xi}} } &= 0 \quad \forall \adjpressure \in L^2(D).
	\end{align}
\end{subequations}
Here, $\vec{A} : \vec{B} = \sum_{j=1}^{d} \sum_{k=1}^{d} A_{j k} B_{j k}$ for two matrices $\vec{A},\vec{B} \in \R^{d\times d}$. The gradient and divergence operators $\nabla$ and $\Divv$ act with respect to the spatial variable only with $\vec{\xi}$ acting as a parameter.

For each shape $u_i$, $i=1,\ldots,N$, we introduce one inequality constraint for a constrained volume, see equation~\eqref{eqn:ModelStokesEquationsVolumeConstraint} and one inequality constraint for a constrained perimeter, see equation~\eqref{eqn:ModelStokesEquationsPerimeterConstraint}. 
The volume of the domain $D_{i}$ is given by
$\vol(D_i) = \intDi{1}$
and the perimeter of $u_i$ is given by
$ \peri(u_i) = \intui{1}.$ 
Now, we suppose there is a deterministic target velocity $\bar{\velocity}$ to be reached on the domain $D$. We would like to determine the optimal placement of shapes that come closest on average to this velocity field. More precisely, we solve the problem
\begin{align}
	\label{eqn:ModelMinimizationProblem}
	\min_{\vu\in B_e^N(S^1,\R^2)} \, \left\lbrace j(\vu) = \int_\Omega { \intD{ \lVert \velocity_{\vec{\xi}(\omega)}(\vx) + \vec{g}_{\vec{\xi}(\omega)}(\vec{x})- \bar{\velocity}(\vec{x}) \rVert_2^2 } \d \pP(\omega)} \right\rbrace
\end{align}
subject to \eqref{eq:weak-formulation} and
\begin{subequations}	
	\label{eq:ModelStokesEquations}
	\begin{align}
		\label{eqn:ModelStokesEquationsVolumeConstraint}
		\vol(D_i)&\geq \underline{\mathcal{V}}_i  & \forall i=1,\ldots,N ,\\
		\label{eqn:ModelStokesEquationsPerimeterConstraint}
		\peri(u_i) &\leq \overline{\mathcal{P}}_{i} & \forall i=1,\ldots,N.
	\end{align}
\end{subequations}
We note that a deterministic model using a tracking-type functional in combination with Stokes flow has been studied in \cite{MR4144267}.

\new{\paragraph{Shape derivative.}}
In the following, we \new{compute the shape derivative of the para\-metrized augmented Lagrangian corresponding to}  the model problem defined by \eqref{eq:weak-formulation}--\eqref{eq:ModelStokesEquations}. We define $\vec{h}\colon B_e^N(S^1,\R^2) \rightarrow \R^{2 N}$ by 
\begin{equation*}
	\vec{h}(\vec{u}) = \begin{pmatrix}
		\vec{h}_{V}(\vu) \\
		\vec{h}_{\mathcal{P}}(\vu)
	\end{pmatrix}  = 
	\begin{pmatrix}
		\left[ \underline{\mathcal{V}}_i - \vol(D_i) \right]_{i\in\{1,\dots,N\}} \\[3pt]
		\left[ \peri(u_i) - \overline{\mathcal{P}}_i \right]_{i\in\{1,\dots,N\}}
	\end{pmatrix},
\end{equation*}
as well as the set $\vec{K}  := \{ \vec{h} \in \R^{2 N}:  h_i \leq 0 \,\, \forall i=1, \ldots, 2 N\}$ and the objective
$J(\vu,\vec{\xi}):=\intD{ \lVert \velocity_{\vec{\xi}}(\vx) + \vec{g}_{\vec{\xi}}(\vec{x})- \bar{\velocity}(\vx) \rVert_2^2 }.$
The parametrized augmented Lagrangian is defined by
\begin{align}
	\label{eqn:augLagrangeFunction}
	\begin{aligned}
		L_{A}(\vu, \vec{\lambda},\vec{\xi};\mu) 	&= J(\vec{u},\vec{\xi})+ \intD{  \nabla \velocity_{\vec{\xi}} : \nabla \adjvelocity_{\vec{\xi}} - \pressure_{\vec{\xi}} \Div{\adjvelocity}_{\vec{\xi}} + \adjpressure_{\vec{\xi}} \Div{\velocity}_{\vec{\xi}} } \\
		&\quad + \frac{\mu}{2} \dist_{\vec{K}} \left(\vec{h}(\vec{u}) + \frac{\vec{\lambda}}{\mu} \right)^2 - \frac{\|\vec{\lambda}\|_2^2}{2 \mu}.
	\end{aligned}
\end{align}

Differentiating the Lagrangian \eqref{eqn:augLagrangeFunction} with respect to $\left(\velocity,\pressure\right)$ and setting it to zero gives the weak form of the adjoint equation: find $\adjvelocity_{\vec{\xi}} \in V(D)$ and $\adjpressure_{\vec{\xi}} \in L^2(D)$ such that 
\begin{subequations}
	\label{eqn:ModelStokesEquationsMomentumAdj}	
	\begin{align}
		\intD{ 2 \tilde{\adjvelocity}^\top \left( \velocity_{\vec{\xi}} + \vec{g}_{\vec{\xi}}  - \bar{\velocity} \right) + \nabla \adjvelocity_{\vec{\xi}} : \nabla \tilde{\adjvelocity}+ \adjpressure_{\vec{\xi}} \Div{\tilde{\adjvelocity}} } &= 0 \quad \forall \tilde{\adjvelocity} \in V(D),\\
		\intD{ \Div{\adjvelocity}_{\vec{\xi}} \,\tilde{\adjpressure} }  & = 0 \quad \forall \tilde{\adjpressure} \in L^2(D).
	\end{align}
\end{subequations}
We define the space $\mathcal{W}(D) = \{ \vec{W} \in H^1(D, \R^2)\colon \vec{W}|_{\Gamma}=0 \}$. We have the shape derivative
\begin{align*}
	\Deriv &L_A (\vu, \vec{\lambda}, \vec{\xi};\mu)\left[\vec{W}\right] \nonumber \\
	=& \int_{D} - \left( \nabla \velocity_{\vec{\xi}} \nabla \vec{W} \right) : \nabla \adjvelocity_{\vec{\xi}} - \left( \nabla \adjvelocity_{\vec{\xi}} \nabla \vec{W} \right) : \nabla \velocity_{\vec{\xi}} + \left( \pressure_{\vec{\xi}} {\nabla \adjvelocity_{\vec{\xi}}}^\top - \adjpressure_{\vec{\xi}} {\nabla \velocity_{\vec{\xi}}}^\top \right) : \nabla \vec{W} \nonumber\\
	&\hphantom{\int_{D_{\vu}}\,} + \Div(\vec{W}) \left( \lVert \velocity_{\vec{\xi}} + \vec{g}_{\vec{\xi}}- \bar{\velocity} \rVert_2^2 + \nabla \velocity_{\vec{\xi}} : \nabla \adjvelocity_{\vec{\xi}} - \pressure_{\vec{\xi}} \Div\adjvelocity_{\vec{\xi}}  + \adjpressure_{\vec{\xi}} \Div \velocity_{\vec{\xi}} \right) \!\dx \nonumber\\
	&+ \mu \left( \left( \vec{h}(\vec{u}) + \frac{\vec{\lambda}}{\mu} \right) - \pi_{\vec{K}} \left(\vec{h}(\vec{u}) + \frac{\vec{\lambda}}{\mu} \right) \right)^\top \nonumber \\ 
	&\hphantom{+}\ \begin{pmatrix}
		\left[\intDi{\Divv{(\vec{W})}} \right]_{i\in\{1,\dots,N\}} \\[3pt]
		\left[\intui{\Divv{(\vec{W})} - \vn^\top \nabla \vec{W} \vn} \right]_{i\in\{1,\dots,N\}}
	\end{pmatrix}, \label{eq:ShapeDerivativeLA}
\end{align*}
where $\left(\velocity_{\vec{\xi}},\pressure_{\vec{\xi}}\right)$ and $\left(\adjvelocity_{\vec{\xi}},\adjpressure_{\vec{\xi}}\right)$ solve the state equation \eqref{eq:weak-formulation} and adjoint equation \eqref{eqn:ModelStokesEquationsMomentumAdj}, respectively. 
	The shape derivative is needed to represent the gradient with respect to the metric under consideration (cf., e.g., \cite{LoayzaGeiersbachWelkerHandbook}).
	As described in \cite{LoayzaGeiersbachWelkerHandbook}, we can use the multi-shape derivative in an ``all-at-once''-approach to compute the multi-shape gradient  with respect to the Steklov--Poincaré metric and the mesh deformation~$\vV = \vV_{\vec{\xi}}$
	all at once 
	by solving
	\begin{equation}
		a(\vec{V}, \vec{W}) = \Deriv L_A (\vu,\vec{\lambda}, \vec{\xi};\mu)[\vec{W}] \quad \forall \vec{W}\in \mathcal{W}(D)\cap \mathcal{C}^\infty(D,\R^2),
		\label{deformatio_equation}
	\end{equation}
	where $a$ is a coercive and symmetric bilinear form. 
	The mesh deformation $\vec{V}$ calculated from (\ref{deformatio_equation}) can be viewed as an extension of the multi-shape gradient~$\vec{v}$ with respect to the Steklov--Poincaré metric to the hold-all domain $D$  (for details we refer the reader to \cite{LoayzaGeiersbachWelkerHandbook}).

	The bilinear form that describes linear elasticity is a common choice for $a$ due to the advantageous effect on the computational mesh (cf.~\cite{Siebenborn2017,Welker2016}), and is selected for the following numerical studies. The Lam\'e parameters are chosen as $\hat{\lambda}=0$ and $\hat{\mu}$ smoothly decreasing from $33$ on $\vu$ to $10$ on $\Gamma$, as obtained by the solution of Poisson's equation on $D$.

	To update the shapes according to Algorithm \ref{alg:algorithmAugLagr}, we need to compute the multi-exponential map.
	This computation is prohibitively expensive in most applications because a calculus of variations problem must be solved or the Christoffel symbols need be known. 
	Therefore, we approximate it using a multi-retraction
	\[
	\mathcal{R}_{\vec{z}^{k,j}}^N\colon T_{\vec{z}^{k,j}}\mathcal{U}^N\to \mathcal{U}^N,\, \vec{v}=(v_1,\dots,v_N)\mapsto (\mathcal{R}_{z^{k,j}_1}v_1,\dots,\mathcal{R}_{z^{k,j}_N}v_N)
	\]
	to update the shape vector $\vec{z}^{k,j}=(z^{k,j}_1,\dots,z^{k,j}_N)$ in each  pair $(j,k)$.
	For each shape $z_i^{k,j}$ we use the retraction in \cite{Geiersbach2021,LoayzaGeiersbachWelkerHandbook,SchulzWelker}:  $\mathcal{R}_{z_i^{k,j}}\colon T_{z_i^{k,j}}\mathcal{U}^i \to \mathcal{U}^i,\, 
	v_i \mapsto  z_i^{k,j}+v_i $ for all $i=1,\dots,N$.
	
	\paragraph{Numerical results.}
	All numerical simulations were performed on the HPC cluster HSUper\footnote{Further information about the technical specifications can be found at \texttt{https://www.hsu-hh.de/hpc/en/hsuper/}.} using the FEniCS toolbox, version 2019.1.0 \cite{AlnaesEtal2015} and Python \new{3.10.10}. The hold-all domain is chosen as $D=(0,1)^2$. We choose $N=3$ shapes inside the hold-all domain, which can be seen on the left-hand side of Figure~\ref{fig:NumericResults_initial_and_target_and_final}. The computational mesh is generated with Gmsh \new{4.11.1} \cite{Geuzaine2009}, which yields 265 line elements for the outer boundary and the interfaces, and 3803 triangular elements as the discretization of $D$. Additionally, a new mesh was automatically generated if the mesh quality\footnote{The mesh quality is measured with the FEniCS function \texttt{MeshQuality.radius\_ratio\_min\_max}.} fell below a threshold of $40\%$. \new{A reevaluation of all relevant values within the optimization (e.g., objective functional and geometrical constraints) after remeshing ensures that optimization can continue to be performed. It has already been observed that this increases the number of optimization iterations (cf., e.g., \cite{Pryymak2023}), but is difficult to avoid due to quickly deteriorating meshes.} The target velocity is shown in Figure~\ref{fig:NumericResults_initial_and_target_and_final} on the right, together with the shapes to obtain the target velocity \new{in white}. Standard Taylor-Hood elements are used.
	
	\begin{figure}[tbp]
		\centering
		\setlength\figureheight{.5\textwidth} 
		\setlength\figurewidth{.5\textwidth}
		\includegraphics{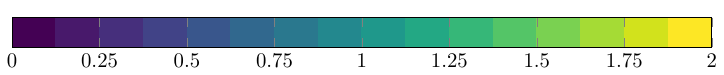}\\%
		\includegraphics{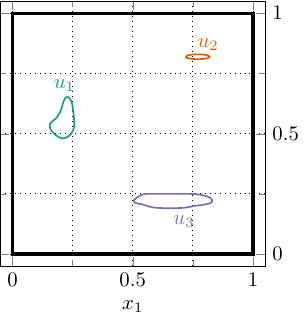}%
		\includegraphics{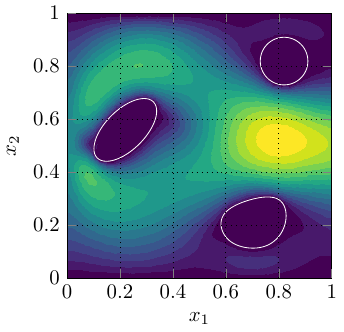}%
		\caption{Shapes $u_1$ (\new{left, green}), $u_2$ (\new{top right, orange}) and $u_3$ (\new{bottom right, blue}) at the start of the stochastic optimization (left) and the magnitude of the target fluid velocity \new{$\| \bar{\velocity} \|_2$} together with the shapes used to obtain the target velocity \new{in white} (right).}
		\label{fig:NumericResults_initial_and_target_and_final}
	\end{figure}
	The values of the geometrical constraints were chosen in accordance with the shapes of the target velocity. The volumes of $D_1$, $D_2$ and $D_3$ were constrained to be at or above $0.035295$, $0.025397$ and $0.036967$, and the perimeters of $u_1$, $u_2$ and $u_3$ to be at or below $0.72630$, $0.56521$ and $0.69796$, respectively. The augmented Lagrangian parameters in Algorithm~\ref{alg:algorithmAugLagr} were initialized to $\vec{\lambda}^1=\vec{0}$, $\mu_1=10$, $\gamma=10$, and $\tau=0.9$. The ball for the projection of Lagrange multipliers was chosen to be $B=[-100,100]^{2N}$. 
	
	\begin{figure}[tbp]
		\centering
		\setlength\figureheight{.2\textwidth} 
		\setlength\figurewidth{.2\textwidth}
		\begin{subfigure}[t]{1.47\figurewidth}%
			\centering%
			\includegraphics{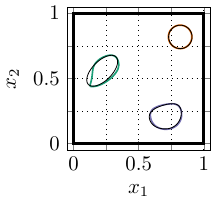}%
			\caption{Seed $964113$.}
			\label{fig:NumericResults_Shapes_seed1}%
		\end{subfigure}%
		\begin{subfigure}[t]{1.1\figurewidth}%
			\centering%
			\includegraphics{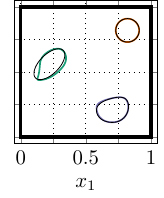}%
			\caption{Seed $454612$.}
			\label{fig:NumericResults_Shapes_seed2}%
		\end{subfigure}%
		\begin{subfigure}[t]{1.1\figurewidth}%
			\centering%
			\includegraphics{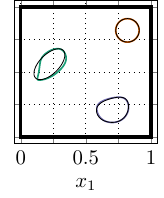}%
			\caption{Seed $421507$.}
			\label{fig:NumericResults_Shapes_seed3}%
		\end{subfigure}%
		\begin{subfigure}[t]{1.1\figurewidth}%
			\centering%
			\includegraphics{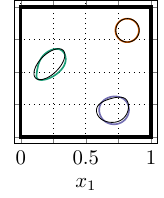}%
			\caption{Seed $107785$.}
			\label{fig:NumericResults_Shapes_seed4}%
		\end{subfigure}%
		\caption{Shapes $u_1$ (\new{left, green}), $u_2$ (\new{top right, orange}) and $u_3$ (\new{bottom right, blue}) \new{after $k=11$ iterations} of stochastic optimization with different seeds and shapes used to obtain the target velocity (black).}
		\label{fig:NumericResults_Shapes}
	\end{figure}
	
	We chose homogenous Dirichlet boundary conditions for the velocity on the top and bottom boundary and on $\vu$ (see Figure~\ref{fig:NumericResults_initial_and_target_and_final}, \new{right}). The inflow profile on the left boundary is modeled as an inhomogenous Dirichlet boundary with $\vec{g}_{\vec{\xi}}(\vec{x}) = ( \kappa(\vec{x},\vec{\xi}), 0)^\top$. The horizontal component is given by the truncated Karhunen-Lo\`eve expansion
	\begin{align*}
		\kappa(\vec{x},\vec{\xi}) = -4 x_2 (x_2-1) + \sum_{\ell=1}^{100} \ell^{-\eta-1/2} \sin(2 \pi \ell (x_2-1/2)) \xi_\ell,
	\end{align*}
	where  $\eta=3.5$ and $\xi_\ell \sim U\!\left[-\frac{1}{2}, \frac{1}{2}\right]$ ($U[a,b]$ being the uniform distribution on the interval $[a,b]$). \new{We used \texttt{numpy.random} from \texttt{numpy} 1.22.4 for the generation of all random values. For this, \texttt{rng=numpy.random.default\_rng(seed)} is used to set the generator and then the random samples are drawn by calling \texttt{rng.uniform(lowerBound, upperBound, shape)}. The lower and upper bound correspond to the bounds of the uniform distribution. The shape of the matrix of random values was set to $(100, m_k)$ yielding $100 \times m_k$ random values per stochastic gradient step, generated row by row. We chose the four different seeds $964113$, $454612$, $421507$ and $107785$. Parallelization of multiple realizations was performed via MPI using \texttt{mpi4py} version 3.1.4, which distributed the matrix to the $m_k$ processes column-wise.} On the right boundary, a homogenous Neumann boundary condition is imposed. The step size is chosen as $t_k=\frac{20}{\mu_k}$, the scaling of which is obtained by tuning (to avoid deterioration of the mesh, especially in the first steps of the inner loop procedure).
	The maximum number of inner loop iterations is \new{chosen to be $N_k=c_1 \cdot 2^k$, with $c_1=4$ or $c_1=25$}. The batch size is increased according to \new{$m_k=c_2 \cdot 2^k$, with $c_2=\frac{1}{2}$ or $c_2=5$}. Each inner loop $k$ requires $m_k\cdot R_k$ solutions of the state equation, the adjoint equation, the Poisson equation for the Lam\'e parameter, and the deformation equation, which becomes computationally expensive for high $k$.
	
	\new{The obtained shapes for $c_1=4$ and $c_2=\frac{1}{2}$ for different seeds are shown in Figure~\ref{fig:NumericResults_Shapes}.}  For all seeds, the \new{top-right} shape $u_2$ looks basically identical to the shape used to obtain the target velocity, however $u_1$ (\new{left}) shows differences \new{at the bottom left and on the right-hand side between different seeds and compared to the shapes for the target velocity}, and $u_3$ has a different \new{left side}. \new{Differences for $u_3$ between the different seeds can also be observed.} We investigate the optimization with the random seed $421507$ further. The remesher is activated after the stochastic gradient step \new{$4$, $9$, $13$, $20$, $26$, $37$, $93$ and $1682$}. In Figure~\ref{fig:NumericResults_objectiveFunctional_and_meshDeformNorm}, the numerical results for objective functional estimate~$\hat{j}=\frac{1}{m_k}\sum_{i=1}^{m_k} J(\vec{z}^{k,j},\vec{\xi}^{k,j,i})$ and the estimate of the $H^1$ norm of the mesh deformation $\widehat{\vV}=\frac{1}{m_k}\sum_{i=1}^{m_k} \vV_{\vec{\xi}^{k,j,i}}$ over cumulative \new{stochastic gradient steps} is provided. Here, even for a comparatively low number of samples per \new{step}, we see a strong decrease in objective functional values initially. The points where the inner loop is stopped due to reaching $R_k$ are denoted by the red vertical dashed lines in the right-hand side plot. At the later stages of the optimization the batch size is increased up to \new{$m_{11}=1024$} for \new{$k=11$}. This yields an increasingly accurate approximation of the mesh deformation and the objective functional value as evidenced by the decreasing variance.
	
	\begin{table}[tbp]
		\caption{\new{Estimate of the norm of the mesh deformation~$\widehat{\vV}$ with a batch size of $m=10024$ and value of the infeasibility measure~$H_k$ at the end of each inner loop using different seeds, $c_1=4$ and $c_2=\frac{1}{2}$.}}
		\small
		\centering
		\setlength\tabcolsep{3pt}
		\renewcommand{\arraystretch}{1.25}
		\new{%
			\begin{tabular}{rr|rrr} 
				\multicolumn{1}{c}{$k$} & \multicolumn{1}{c|}{$R_k$} & \multicolumn{1}{c}{$\| \widehat{\vV} \|_{H^1}$}  & $\mu_k$ & \multicolumn{1}{c}{$H_k$} \\
				\hline
				$1$    & $3$   & $8.751 \cdot 10^{-3}$ & $10$ & $3.853 \cdot 10^{-2}$ \\
				$2$   & $8$  & $5.765 \cdot 10^{-3}$ & $10$ & $3.363 \cdot 10^{-2}$ \\
				$3$   & $4$  & $1.002 \cdot 10^{-2}$ & $10$ & $3.107 \cdot 10^{-2}$ \\
				$4$   & $28$  & $1.372 \cdot 10^{-2}$ & $10$ & $2.621 \cdot 10^{-2}$ \\
				$5$   & $23$  & $1.653 \cdot 10^{-2}$ & $100$ & $2.004 \cdot 10^{-2}$ \\
				$6$   & $109$  & $1.643 \cdot 10 ^{-2}$ & $100$ & $5.905 \cdot 10^{-3}$ \\
				$7$   & $68$ & $2.795 \cdot 10^{-3}$ & $100$ & $4.725 \cdot 10^{-3}$ \\
				$8$   & $490$ & $5.361 \cdot 10^{-3}$ & $100$ & $4.601 \cdot 10^{-3}$ \\
				$9$   & $1918$ & $1.153 \cdot 10^{-3}$ & $1000$ & $1.935 \cdot 10^{-3}$ \\
				$10$  & $140$ & $1.063 \cdot 10^{-3}$ & $1000$ & $9.930 \cdot 10^{-4}$ \\
				$11$  & $3617$ & $8.900 \cdot 10^{-4}$ & $1000$ & $2.721 \cdot 10^{-4}$
			\end{tabular}%
		}%
		\hfill%
		\new{%
			\begin{tabular}{rr|rrr} 
				\multicolumn{1}{c}{$k$} & \multicolumn{1}{c|}{$R_k$} & \multicolumn{1}{c}{$\| \widehat{\vV} \|_{H^1}$}  & $\mu_k$ & \multicolumn{1}{c}{$H_k$} \\
				\hline
				$1$    & $5$   & $7.905 \cdot 10^{-3}$ & $10$ & $3.748 \cdot 10^{-2}$ \\
				$2$   & $6$  & $2.209 \cdot 10^{-2}$ & $10$ & $3.388 \cdot 10^{-2}$ \\
				$3$   & $26$  & $1.452 \cdot 10^{-2}$ & $10$ & $2.948 \cdot 10^{-2}$ \\
				$4$   & $20$  & $1.921 \cdot 10^{-2}$ & $100$ & $2.407 \cdot 10^{-2}$ \\
				$5$   & $120$  & $1.731 \cdot 10^{-2}$ & $100$ & $6.487 \cdot 10^{-3}$ \\
				$6$   & $65$  & $2.583 \cdot 10 ^{-3}$ & $100$ & $4.968 \cdot 10^{-3}$ \\
				$7$   & $97$ & $4.345 \cdot 10^{-3}$ & $100$ & $4.856 \cdot 10^{-3}$ \\
				$8$   & $40$ & $3.236 \cdot 10^{-3}$ & $1000$ & $2.424 \cdot 10^{-3}$ \\
				$9$   & $999$ & $1.339 \cdot 10^{-3}$ & $1000$ & $1.064 \cdot 10^{-3}$ \\
				$10$  & $2174$ & $7.278 \cdot 10^{-4}$ & $1000$ & $1.689 \cdot 10^{-4}$ \\
				$11$  & $7208$ & $8.380 \cdot 10^{-4}$ & $1000$ & $1.399 \cdot 10^{-4}$
			\end{tabular}%
		}%
		\begin{center}
			Seed $964113$ (left) and seed $454612$ (right)
		\end{center}
		\vspace*{.4cm}
		\new{%
			\begin{tabular}{rr|rrr} 
				\multicolumn{1}{c}{$k$} & \multicolumn{1}{c|}{$R_k$} & \multicolumn{1}{c}{$\| \widehat{\vV} \|_{H^1}$}  & $\mu_k$ & \multicolumn{1}{c}{$H_k$} \\
				\hline
				$1$    & $6$   & $6.527 \cdot 10^{-3}$ & $10$ & $3.691 \cdot 10^{-2}$ \\
				$2$   & $13$  & $4.974 \cdot 10^{-3}$ & $10$ & $2.983 \cdot 10^{-2}$ \\
				$3$   & $19$  & $3.395 \cdot 10^{-3}$ & $10$ & $2.039 \cdot 10^{-2}$ \\
				$4$   & $37$  & $2.560 \cdot 10^{-3}$ & $10$ & $6.573 \cdot 10^{-3}$ \\
				$5$   & $66$  & $1.556 \cdot 10^{-3}$ & $10$ & $3.504 \cdot 10^{-3}$ \\
				$6$   & $51$  & $2.132 \cdot 10 ^{-2}$ & $10$ & $5.394 \cdot 10^{-3}$ \\
				$7$   & $228$ & $4.940 \cdot 10^{-4}$ & $100$ & $1.295 \cdot 10^{-3}$ \\
				$8$   & $939$ & $6.495 \cdot 10^{-4}$ & $100$ & $8.305 \cdot 10^{-4}$ \\
				$9$   & $1828$ & $8.367 \cdot 10^{-4}$ & $100$ & $6.427 \cdot 10^{-4}$ \\
				$10$  & $2321$ & $9.144 \cdot 10^{-4}$ & $100$ & $4.940 \cdot 10^{-4}$ \\
				$11$  & $4299$ & $3.715 \cdot 10^{-4}$ & $100$ & $3.239 \cdot 10^{-4}$
			\end{tabular}%
		}%
		\hfill%
		\new{%
			\begin{tabular}{rr|rrr} 
				\multicolumn{1}{c}{$k$} & \multicolumn{1}{c|}{$R_k$} & \multicolumn{1}{c}{$\| \widehat{\vV} \|_{H^1}$}  & $\mu_k$ & \multicolumn{1}{c}{$H_k$} \\
				\hline
				$1$    & $1$   & $1.001 \cdot 10^{-2}$ & $10$ & $3.968 \cdot 10^{-2}$ \\
				$2$   & $4$  & $1.127 \cdot 10^{-2}$ & $10$ & $3.711 \cdot 10^{-2}$ \\
				$3$   & $26$  & $1.571 \cdot 10^{-2}$ & $10$ & $3.248 \cdot 10^{-2}$ \\
				$4$   & $54$  & $1.782 \cdot 10^{-2}$ & $100$ & $1.763 \cdot 10^{-2}$ \\
				$5$   & $10$  & $1.855 \cdot 10^{-2}$ & $100$ & $1.440 \cdot 10^{-2}$ \\
				$6$   & $235$  & $1.795 \cdot 10 ^{-2}$ & $100$ & $7.908 \cdot 10^{-3}$ \\
				$7$   & $27$ & $1.508 \cdot 10^{-3}$ & $100$ & $5.186 \cdot 10^{-3}$ \\
				$8$   & $576$ & $7.635 \cdot 10^{-4}$ & $100$ & $4.502 \cdot 10^{-3}$ \\
				$9$   & $593$ & $8.051 \cdot 10^{-4}$ & $100$ & $4.001 \cdot 10^{-3}$ \\
				$10$  & $182$ & $8.117 \cdot 10^{-4}$ & $100$ & $3.520 \cdot 10^{-3}$ \\
				$11$  & $8014$ & $8.577 \cdot 10^{-4}$ & $100$ & $2.851 \cdot 10^{-3}$
			\end{tabular}%
		}%
		\begin{center}
			Seed $421507$ (left) and seed $107785$ (right)
		\end{center}
		\label{tab:results}%
	\end{table}%
	
	\begin{figure}[tbp]
		\centering
		\setlength\figureheight{5cm}%
		\setlength\figurewidth{.46\textwidth}%
		\includegraphics{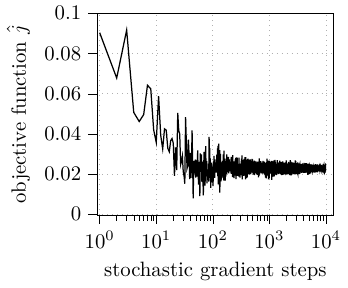}%
		\includegraphics{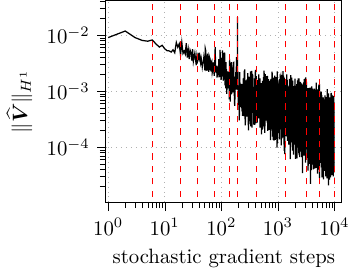}%
		\caption{Objective functional (left) and $H^1$ norm of the mesh deformation (right) as a function of cumulative stochastic gradient steps \new{using seed $421507$, $c_1=4$ and $c_2=\frac{1}{2}$}. The changes of augmented Lagrange parameters are indicated with a red, dashed, vertical line.}
		\label{fig:NumericResults_objectiveFunctional_and_meshDeformNorm}
	\end{figure}
	
	We provide the numerical results at the end of each inner loop \new{for different seeds, $c_1$ and $c_2$} in \new{Tables~\ref{tab:results} and~\ref{tab:results_higher_N_k_m_k}. Here, we present the number of iterations until random stopping~$R_k$, the $H^1$ norm of the mesh deformation for each $k$, which is estimated using the seed $883134$ and a (larger) sample size of $m=10024$ as $\widehat{\vV}= \frac{1}{m} \sum_{i=1}^m \vV_{{\vec{\xi}}^i}$, and the infeasibility measure~$H_k$. Different seeds (Table~\ref{tab:results}) behave differently regrading mesh deformation norm estimate, penalty factor and infeasibility, however the mesh deformation norm estimate and the infeasibility measure were overall reduced by orders of magnitude. We attribute the increases in these values in between to the effect of the randomness on the stochastic gradient. Using larger batch sizes (Table~\ref{tab:results_higher_N_k_m_k}, left) yielded lower mesh deformation norms at a significantly increased computational cost, indicating a very strong influence of the randomness on the objective functional that can be reduced by larger sample sizes, cf. also Figure~\ref{fig:NumericResults_objectiveFunctional_and_meshDeformNorm}. An increased iteration limit~$N_k$ (Table~\ref{tab:results_higher_N_k_m_k}, right) did not seem to improve the result, which is expected due to the strong influence of the randomness on the objective functional.}

	\begin{table}[tbp]
		\caption{\new{Estimate of the norm of the mesh deformation~$\widehat{\vV}$ with a batch size of $m=10024$ and value of the infeasibility measure~$H_k$ at the end of each inner loop using seed $421507$ and different $c_1$ and $c_2$.}}
		\small
		\centering
		\setlength\tabcolsep{3pt}
		\renewcommand{\arraystretch}{1.25}
		\new{%
			\begin{tabular}{rr|rrr}   
				\multicolumn{1}{c}{$k$} & \multicolumn{1}{c|}{$R_k$} & \multicolumn{1}{c}{$\| \widehat{\vV} \|_{H^1}$}  & $\mu_k$ & \multicolumn{1}{c}{$H_k$} \\
				\hline
				$1$   & $39$   & $2.719 \cdot 10^{-3}$ & $10$ & $2.619 \cdot 10^{-2}$ \\
				$2$   & $86$  & $2.432 \cdot 10^{-3}$ & $10$ & $1.104 \cdot 10^{-2}$ \\
				$3$   & $151$  & $8.049 \cdot 10^{-3}$ & $10$ & $7.922 \cdot 10^{-3}$ \\
				$4$   & $324$  & $1.516 \cdot 10^{-3}$ & $10$ & $3.861 \cdot 10^{-3}$ \\
				$5$   & $259$  & $2.472 \cdot 10^{-3}$ & $10$ & $4.508 \cdot 10^{-3}$ \\
				$6$   & $157$  & $9.430 \cdot 10 ^{-4}$ & $100$ & $6.463 \cdot 10^{-4}$ \\
				$7$   & $3085$ & $1.668 \cdot 10^{-3}$ & $100$ & $3.163 \cdot 10^{-4}$ \\
				$8$   & $512$ & $8.117 \cdot 10^{-4}$ & $100$ & $2.570 \cdot 10^{-4}$ \\
				$9$   & $8894$ & $4.183 \cdot 10^{-4}$ & $100$ & $1.232 \cdot 10^{-4}$ \\
				$10$  & $6623$ & $8.896 \cdot 10^{-3}$ & $100$ & $1.984 \cdot 10^{-4}$ \\
				$11$  & $11021$ & $7.836 \cdot 10^{-4}$ & $1000$ & $1.950 \cdot 10^{-5}$
			\end{tabular}%
		}%
		\hfill%
		\new{%
			\begin{tabular}{rr|rrr}   
				\multicolumn{1}{c}{$k$} & \multicolumn{1}{c|}{$R_k$} & \multicolumn{1}{c}{$\| \widehat{\vV} \|_{H^1}$}  & $\mu_k$ & \multicolumn{1}{c}{$H_k$} \\
				\hline
				$1$   & $6$   & $6.486 \cdot 10^{-3}$ & $10$ & $3.702 \cdot 10^{-2}$ \\
				$2$   & $13$  & $5.522 \cdot 10^{-3}$ & $10$ & $3.027 \cdot 10^{-2}$ \\
				$3$   & $30$  & $2.435 \cdot 10^{-3}$ & $10$ & $1.618 \cdot 10^{-2}$ \\
				$4$   & $7$  & $2.519 \cdot 10^{-3}$ & $10$ & $1.310 \cdot 10^{-2}$ \\
				$5$   & $76$  & $2.721 \cdot 10^{-3}$ & $10$ & $6.572 \cdot 10^{-3}$ \\
				$6$   & $31$  & $3.534 \cdot 10 ^{-2}$ & $10$ & $1.791 \cdot 10^{-3}$ \\
				$7$   & $152$ & $1.129 \cdot 10^{-4}$ & $10$ & $1.717 \cdot 10^{-3}$ \\
				$8$   & $678$ & $6.144 \cdot 10^{-4}$ & $100$ & $1.229 \cdot 10^{-3}$ \\
				$9$   & $1383$ & $1.978 \cdot 10^{-4}$ & $100$ & $8.335 \cdot 10^{-4}$ \\
				$10$  & $2662$ & $2.120 \cdot 10^{-4}$ & $100$ & $5.980 \cdot 10^{-4}$ \\
				$11$  & $6893$ & $1.691 \cdot 10^{-4}$ & $100$ & $4.393 \cdot 10^{-4}$
			\end{tabular}%
			\begin{center}
				$c_1=25$, $c_2=\frac{1}{2}$ (left) and $c_1=4$, $c_2=5$ (right)
			\end{center}
		}%
		\label{tab:results_higher_N_k_m_k}%
	\end{table}%

	\new{As an additional numerical experiment, we investigated the influence of the choice of $B$ for the projection of Lagrange multipliers. Instead of $B=[-100,100]^{2N}$ we chose $B=[-0.1,0.1]^{2N}$. The results are provided in Table~\ref{tab:results_smaller_ball}. The batch size and maximum number of inner loop iterations match those in Table~\ref{tab:results}, bottom left. Therefore, the random samples were exactly the same in both cases, but the optimization problem changes since the Lagrange multipliers are different. We did not see any notable improvement in performance by choosing the smaller set. The numerical results indicate that the infeasibility measure is reduced more slowly while the norm of the mesh deformation decreases slightly more rapidly. For higher $k$, we observed a stronger influence of the uncertainty on the mesh deformation norm.}
	
	\begin{table}[tbp]
		\caption{\new{Estimate of the norm of the mesh deformation~$\widehat{\vV}$ with a batch size of $m=10024$ and value of the infeasibility measure~$H_k$ at the end of each inner loop using the smaller ball $B=\left[ -0.1, 0.1 \right]^{2N}$ for the projection of Lagrange multipliers, seed $421507$, $c_1=4$ and $c_2=\frac{1}{2}$.}}
		\small
		\centering
		\setlength\tabcolsep{3pt}
		\renewcommand{\arraystretch}{1.25}
		\new{%
			\begin{tabular}{rr|rrr}   
				\multicolumn{1}{c}{$k$} & \multicolumn{1}{c|}{$R_k$} & \multicolumn{1}{c}{$\| \widehat{\vV} \|_{H^1}$}  & $\mu_k$ & \multicolumn{1}{c}{$H_k$} \\
				\hline
				$1$   & $6$   & $6.363 \cdot 10^{-3}$ & $10$ & $3.691 \cdot 10^{-2}$ \\
				$2$   & $13$  & $4.579 \cdot 10^{-3}$ & $10$ & $3.066 \cdot 10^{-2}$ \\
				$3$   & $19$  & $2.703 \cdot 10^{-3}$ & $10$ & $2.477 \cdot 10^{-2}$ \\
				$4$   & $37$  & $2.126 \cdot 10^{-3}$ & $10$ & $1.779 \cdot 10^{-2}$ \\
				$5$   & $66$  & $6.739 \cdot 10^{-4}$ & $10$ & $1.125 \cdot 10^{-2}$ \\
				$6$   & $51$  & $9.022 \cdot 10 ^{-4}$ & $10$ & $8.407 \cdot 10^{-3}$ \\
				$7$   & $228$ & $2.271 \cdot 10^{-4}$ & $10$ & $4.805 \cdot 10^{-3}$ \\
				$8$   & $939$ & $5.354 \cdot 10^{-4}$ & $10$ & $4.314 \cdot 10^{-3}$ \\
				$9$   & $1828$ & $1.611 \cdot 10^{-3}$ & $10$ & $4.583 \cdot 10^{-3}$ \\
				$10$  & $2321$ & $7.424 \cdot 10^{-4}$ & $100$ & $1.490 \cdot 10^{-3}$ \\
				$11$  & $4299$ & $4.412 \cdot 10^{-3}$ & $100$ & $1.480 \cdot 10^{-3}$
			\end{tabular}%
		}%
		\label{tab:results_smaller_ball}%
	\end{table}%

	\new{
		\section{Conclusion}
		\label{sec:conclusion}
		In this paper, we introduced a novel method for solving constrained optimization problems under uncertainty, where the optimization variable belongs to a Riemannian (shape) manifold. The objective function is formulated as an expectation and the constraints are deterministic. Our work is motivated by applications in PDE-constrained shape optimization, where uncertainty enters the problem in the form of a random PDE, and geometric constraints are introduced to avoid trivial solutions. The optimization variable---the shape---is understood as an element of a Riemannian shape manifold.}
	
	\new{Using the framework of Riemannian manifolds allows us to rigorously prove the convergence of our method, which we call the stochastic augmented Lagrangian method. This algorithm consists of a batch stochastic gradient method with random stopping in an inner loop, combined with an augmented Lagrangian method in an outer loop. The inherently nonconvex character of our underlying application is the reason for introducing random stopping and it allows us to prove convergence rates in expectation even in the absence of convexity. A price that is paid for the guaranteed convergence rates is that the inner loop procedure becomes increasingly expensive. While this is a disadvantage, this still outperforms the standard approach used in sample average approximation, where a one-time sample is taken and the corresponding problem is solved using all samples. The stochastic approximation approach used here dynamically samples over the course of optimization, allowing us to use dramatically fewer samples, especially in the first iterations of the augmented Lagrangian procedure. To our knowledge, our method is the first to solve this kind of shape optimization problem under uncertainty. 
		Since this is quite new, the results of this paper leave space for future research. In particular, there are a few open questions from differential geometry that are outside the scope of the paper but that came up while formulating our theory. It is still unclear whether Assumption \ref{assump:manifold} is satisfied for the manifold used in our application. In particular, we require connectivity and the existence of a bounded injectivity radius of the shape space under consideration. 
	}

\par\addvspace{\bigskipamount}
	
\begin{acknowledgements}
	\noindent \textbf{Acknowledgements} This work has been partly supported by the German Research Foundation (DFG) within the priority program SPP~1962 under contract number WE~6629/1-1 and 
	by the state of Hamburg (Germany) within the Landesforschungsförderung under project ``Simulation-Based Design Optimization of Dynamic Systems Under Uncertainties'' (SENSUS) with project number LFF-GK11. Computational resources (HPC cluster HSUper) have been provided by the
	project hpc.bw, funded by dtec.bw -- Digitalization and Technology Research Center of the Bundeswehr. dtec.bw is funded by the European Union -- NextGenerationEU.
\end{acknowledgements}

\begin{data}
	\noindent \textbf{Data Availability Statement} Most of the data generated or analyzed during this study are included in this
	published article. Any additional information, \new{including the code for simulations}, is available from the corresponding author upon reasonable request.
\end{data}
\bibliographystyle{hplain_custom}
\bibliography{sublibrary.bib}

\end{document}